\documentclass[12pt]{scrartcl}

\usepackage[]{amsmath, amssymb,amsfonts,amsthm,mathtools,braket,color,enumerate}

\usepackage[margin=1in]{geometry}
\usepackage{subcaption}
\usepackage{hhline}
\usepackage{url}
\usepackage{here}
\usepackage{tikz-cd}
\usepackage{tikz}
\usepackage{hyperref}
\usepackage{tkz-graph}
\hypersetup{
  colorlinks   = true, 
  urlcolor     = blue, 
  linkcolor    = blue, 
  citecolor   = red 
}
\usetikzlibrary{decorations}
\usetikzlibrary{arrows}
\tikzstyle{v} = [circle, draw, inner sep=2pt, minimum size=3pt, fill=black]
\tikzstyle{l} = [rectangle, draw, rounded corners]

\theoremstyle{plain}
\newtheorem{theorem}{Theorem}[section]
\newtheorem{lemma}[theorem]{Lemma}
\newtheorem{proposition}[theorem]{Proposition}
\newtheorem{corollary}[theorem]{Corollary}

\newtheorem*{theorem*}{Theorem}
\newtheoremstyle{named}{}{}{\itshape}{}{\bfseries}{.}{.3em}{#1\thmnote{ #3}}
\theoremstyle{named}

\theoremstyle{definition}
\newtheorem{definition}[theorem]{Definition}
\newtheorem{conjecture}[theorem]{Conjecture}
\newtheorem{example}[theorem]{Example}

\DeclareMathOperator{\Cox}{Cox}

\DeclareMathOperator{\Shi}{Shi}

\DeclareMathOperator{\Ish}{Ish}
\DeclareMathOperator{\Der}{Der}
\newcommand{\Mat}{\operatorname{Mat}}
\newcommand{\A}{\mathcal{A}}
\newcommand{\Hh}{\mathcal{H}}
\newcommand{\mcS}{\mathcal{S}}
\newcommand{\I}{\mathcal{I}}
\newcommand{\B}{\mathcal{B}}
\newcommand{\D}{\mathcal{D}}

\newcommand{\M}{\mathcal{M}}
\newcommand{\Z}{\mathbb{Z}}
\newcommand{\K}{\mathbb{K}}
\newcommand{\R}{\mathbb{R}}

\newcommand{\rk}{\operatorname{rank}} 

\newcommand{\cc}{\mathbf{c}} 
\newcommand{\tbf}{\textbf} 
\newcommand{\quasi}{\operatorname{quasi}}

\newcommand{\C}{\mathbb{C}}
\newcommand{\lcm}{\operatorname{lcm}}

\allowdisplaybreaks

\begin{document}

\title{A type $B$ analog of Ish arrangement}

\author{
Tan Nhat Tran
\thanks{
Institut f\"ur Algebra, Zahlentheorie und Diskrete Mathematik, Fakult\"at f\"ur Mathematik und Physik, Leibniz Universit\"at Hannover, Welfengarten 1, D-30167 Hannover, Germany.
E-mail address: tan.tran@math.uni-hannover.de. 
}\and
Shuhei Tsujie
\thanks{Department of Mathematics, Hokkaido University of Education, Asahikawa, Hokkaido 070-8621, Japan. 
E-mail address: tsujie.shuhei@a.hokkyodai.ac.jp.}
}

\date{\today}

\maketitle

\begin{abstract}
The Shi arrangement due to Shi (1986) and the Ish arrangement due to Armstrong (2013) are deformations of the type $A$ Coxeter arrangement that share many common properties. Motivated by a question of Armstrong and Rhoades since 2012 to seek for Ish arrangements of other types, in this paper we introduce an Ish arrangement of type $B$. We study this Ish arrangement through various aspects similar to as known in type $A$ with a main emphasis on freeness and supersolvability. Our method is based on the concept of $\psi$-digraphic arrangements recently introduced due to Abe and the authors with a  type $B$ extension.
\end{abstract}


   \tableofcontents
   
\section{Introduction}
\quad
Let $V= \mathbb{R}^{\ell} $ be a finite-dimensional real vector space. 
Let $ \{x_{1}, \dots, x_{\ell}\} $ be a basis for the dual space $V^{\ast} $. 
Our discussion starts with the \textbf{Shi arrangement} $ \Shi(A_{\ell-1}) $ due to Shi \cite[Chapter 7]{Shi86},  and the \textbf{Ish arrangement} $ \Ish(A_{\ell-1}) $ due to Armstrong \cite{Arms13} defined as
\begin{equation}
\label{eq:SI}
\begin{aligned}
\Shi(A_{\ell-1}) &\coloneqq\Cox(A_{\ell-1}) \cup  \Set{ x_{i}-x_{j}= 1 | 1 \leq i < j \leq \ell}, \\
 \Ish(A_{\ell-1}) &\coloneqq\Cox(A_{\ell-1}) \cup \Set{ x_{1}-x_{j} = i | 1 \leq i < j \leq \ell},
\end{aligned}
\end{equation}
where $\Cox(A_{\ell-1})  \coloneqq \Set{ x_{i}-x_{j}=0 | 1 \leq i < j \leq \ell} $ is the  \textbf{Coxeter arrangement of type $A_{\ell-1}$}. 
The Ish arrangement is defined in the inspiration of the so-called \textbf{combinatorial symmetry} \cite{AR12} 
\begin{equation}
\label{eq:CS}
\{x_{i}-x_{j}= 1\} \longleftrightarrow \{ x_{1}-x_{j} = i\} \quad (1 \leq i < j \leq \ell).
\end{equation}

This correspondence is simply a set bijection yet led to many common properties of the Shi and Ish arrangements from different aspects, among others, the perspective of freeness and supersolvability. 
See \S\ref{subsec:free} and \S\ref{subsec:SS} for the definitions of free and supersolvable arrangements. 

We call any property that the Shi and Ish arrangements share in common a ``\tbf{Shi/Ish duality}". 
The following Shi/Ish dualities hold: 

\begin{theorem}[\cite{Headley97, A96, Arms13, AR12}]\label{Shi/Ish characteristic polynomial}
The arrangements $\Shi(A_{\ell-1})$ and $\Ish(A_{\ell-1})$ have the same characteristic polynomial $t(t-\ell)^{\ell-1}$. 
\end{theorem}

\begin{theorem}[\cite{Atha98, AST17,Yo04}]\label{Shi/Ish freeness}
The cones over $\Shi(A_{\ell-1})$ and $\Ish(A_{\ell-1})$ are both free. 
Moreover, the cone over $\Ish(A_{\ell-1})$ is supersolvable. 
\end{theorem}

\begin{theorem}[\cite{AR12,LRW14}]\label{Shi/Ish chamber}
The arrangements $\Shi(A_{\ell-1})$ and $\Ish(A_{\ell-1})$ have the same number of regions (connected components of the arrangement's complement) with $c$ ceilings and $d$ degrees of freedom for any nonnegative integers $c,d$. 
\end{theorem}

It was questioned by Armstrong and Rhoades  \cite[\S5.3(6)]{AR12} to define and study Ish arrangements for root systems of other types. 
The main purpose of this paper is to introduce, for the first time, an Ish arrangement for  type $B$ root system which has  properties similar to the ones in Theorems \ref{Shi/Ish characteristic polynomial}, \ref{Shi/Ish freeness} and \ref{Shi/Ish chamber}. 
Here is our main definition. 

\begin{definition}
\label{def:Shi-Ish-B}
The \tbf{Shi arrangement $\Shi(B_\ell)$} and the \tbf{Ish arrangement $\Ish(B_\ell)$ of type $B_{\ell}$} are defined by 
\begin{equation*}
\label{eq:BSI}
\begin{aligned}
\Shi(B_\ell) \coloneqq&\Cox(B_\ell)  \cup  \Set{ x_{i} = 1 | 1 \leq i  \leq \ell } \cup  \Set{ x_{i}- x_{j}= 1 | 1 \leq i < j \leq \ell} \\ 
& \qquad\cup  \Set{x_{i}+x_{j}= 1 | 1 \leq i < j \leq \ell}, \\[2mm]
\Ish(B_\ell) \coloneqq& \Set{x_{i}\pm x_{j} = 0 | 1 \leq i < j \leq \ell}  \cup \Set{x_{i} = a | 1 \leq i \leq \ell, \,i-\ell \leq a \leq \ell-i+1 }\\
=& \Cox(B_\ell)  \cup  \Set{ x_{i} = 1 |1 \leq i  \leq \ell} \cup  \Set{ x_{i} = \ell+2-j  | 1 \leq i < j \leq \ell}\\ 
& \qquad\cup  \Set{ x_{i} =  -(\ell+1-j )| 1 \leq i < j \leq \ell},
\end{aligned}
\end{equation*}
where $\Cox(B_\ell) \coloneqq \Set{ x_{i}\pm x_{j}=0 | 1 \leq i < j \leq \ell} \cup \Set{ x_{i} = 0 |  1 \leq i  \leq \ell} $ is the  \textbf{Coxeter arrangement of type $B_\ell$}. 
(The case when $\ell=2$ is depicted in Figure \ref{fig:Shi Ish B2}.)
\end{definition}

  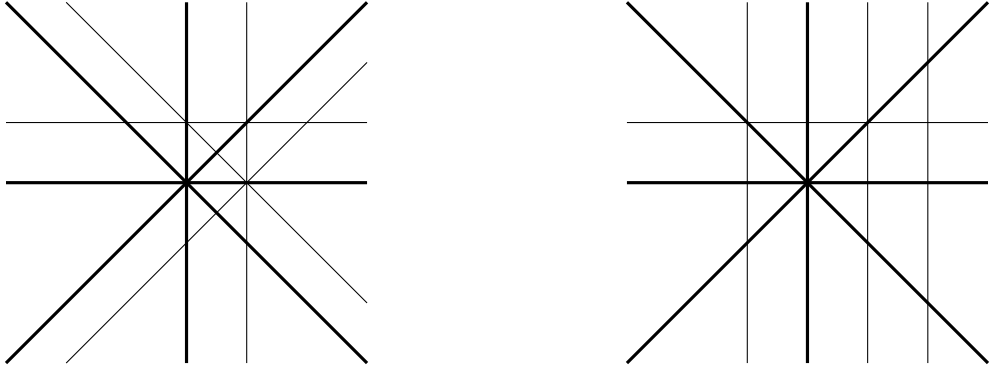
\begin{figure}[htbp!]
\centering
\begin{subfigure}{.5\textwidth}
  \centering
\begin{tikzpicture}[scale=.8]
\draw[very thick] (-3,0)--(3,0);
\draw[very thick] (0,-3)--(0,3);
\draw[very thick] (-3,-3)--(3,3);
\draw[very thick] (-3,3)--(3,-3);
\draw (-3,1)--(3,1);
\draw (1,-3)--(1,3);
\draw(-2,-3)--(3,2);
\draw (-2,3)--(3,-2);
\end{tikzpicture}
 \end{subfigure}%
\begin{subfigure}{.5\textwidth}
  \centering
\begin{tikzpicture}[scale=.8]
\draw[very thick] (-3,0)--(3,0);
\draw[very thick] (0,-3)--(0,3);
\draw[very thick] (-3,-3)--(3,3);
\draw[very thick] (-3,3)--(3,-3);
\draw (-1,-3)--(-1,3);
\draw (1,-3)--(1,3);
\draw (2,-3)--(2,3);
\draw (-3,1)--(3,1);
\end{tikzpicture}
 \end{subfigure}%
\caption{$\Shi(B_2)$ and $\Ish(B_2)$}
\label{fig:Shi Ish B2}
\end{figure}

Let us describe a relation between the Coxeter, Shi and Ish arrangements of types $A$ and $B$. 
For an arrangement $\A$, let $ \mathcal{A}^{\mathrm{ess}} $  denote its \textbf{essentialization} (see \S\ref{subsec:free}). 
Apply the transformation $x_1 \mapsto x_1$, $x_i \mapsto x_1-x_{\ell+2-i}$ $(2 \le i \le \ell +1)$ we know that
\begin{align*}
\Cox(A_{\ell})^{\mathrm{ess}} &\simeq \Set{x_{i}-x_{j} = 0 | 1 \leq i < j \leq \ell} \cup \Set{x_{i} = 0 | 1 \leq i \leq \ell}, \\
\Shi(A_{\ell})^{\mathrm{ess}} &\simeq \Set{x_{i}-x_{j} = 0,1 | 1 \leq i < j \leq \ell} \cup \Set{x_{i} = 0,1 | 1 \leq i \leq \ell}, \\
\Ish(A_{\ell})^{\mathrm{ess}} &\simeq \Set{x_{i}-x_{j} = 0 | 1 \leq i < j \leq \ell} \cup \Set{x_{i} = \ell+2-j | 1 \leq i < j \leq \ell+2},
\end{align*}
where ``$\simeq$" means \tbf{affine equivalence} (see \S\ref{subsec:SS}). 
The combinatorial symmetry \eqref{eq:CS} between $\Shi(A_{\ell})^{\mathrm{ess}}$ and $\Ish(A_{\ell})^{\mathrm{ess}}$ can now be written as
\begin{equation}
\label{eq:RCS}
\begin{aligned}
\{x_{i} = 1\} &\longleftrightarrow \{x_{i} = 1\} \quad (1 \leq i \leq \ell), \\
\{x_{i}-x_{j}= 1\} &\longleftrightarrow \{ x_{i} = \ell+2-j\}  \quad (1 \leq i < j \leq \ell). 
\end{aligned}
\end{equation}

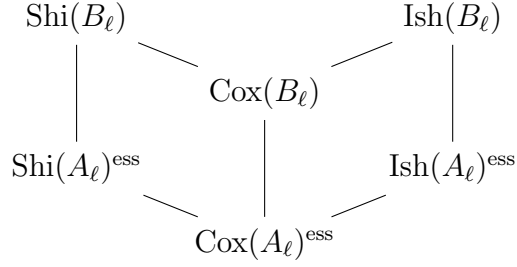
\begin{figure}[htbp!]
\centering
\begin{tikzpicture}
\draw (0,0) node(CA){$\Cox(A_{\ell})^{\mathrm{ess}}$};
\draw (0,2) node(CB){$\Cox(B_{\ell})$};
\draw (-2.5,1) node(SA){$\Shi(A_{\ell})^{\mathrm{ess}}$};
\draw (-2.5,3) node(SB){$\Shi(B_{\ell})$};
\draw ( 2.5,1) node(IA){$\Ish(A_{\ell})^{\mathrm{ess}}$};
\draw ( 2.5,3) node(IB){$\Ish(B_{\ell})$};
\draw (CA)--(CB);
\draw (SA)--(SB);
\draw (IA)--(IB);
\draw (SA)--(CA)--(IA);
\draw (SB)--(CB)--(IB);
\end{tikzpicture}
\caption{The Coxeter, Shi and Ish arrangements of types $A$ and $B$ ordered by inclusion.}
\label{fig:Hasse shi ish}
\end{figure}

By this way, $\Cox(B_{\ell})$ and $\Ish(A_{\ell})^{\mathrm{ess}}$ can be regarded as subarrangements of $\Ish(B_{\ell})$ in the same way that  $\Cox(B_{\ell})$ and $\Shi(A_{\ell})^{\mathrm{ess}}$ do for $\Shi(B_{\ell})$. 
Figure \ref{fig:Hasse shi ish} shows the Hasse diagram of these arrangements ordered by inclusion.  

In this paper, we extend Theorems \ref{Shi/Ish characteristic polynomial} and \ref{Shi/Ish freeness} to $\Shi(B_{\ell})$ and $\Ish(B_{\ell})$. 
More precisely, we show that their associated cones share the characteristic polynomial and freeness, and the cone over  $\Ish(B_{\ell})$ is supersolvable (Corollary \ref{cor:Shi/Ish freeness and char poly}).  
An extension of Theorem \ref{Shi/Ish chamber} to the regions of $\Shi(B_\ell)$ and $\Ish(B_\ell)$ will be given in the upcoming paper  \cite{NTY} of Numata, Yazawa and the second author.

More generally, we introduce three families of arrangements related to $\Shi(B_{\ell})$ and $\Ish(B_{\ell})$ in analogy to the case of type $A$, and investigate their freeness and supersolvability:
\begin{itemize}
\item \tbf{type $B$ Shi descendants} which generalize both Shi and Ish arrangements (\S\ref{subsec:BShi-des}), 
\item \tbf{type $B$ $N$-Ish arrangements} which generalize Ish arrangement (\S\ref{subsec:BN-Ish}), 
\item \tbf{type $B$ deleted Shi and Ish arrangements} which provide ``partial" interpolation between Coxeter, Shi and Ish arrangements (\S\ref{subsec:Bdelete}). 
\end{itemize}

Finally, in \S\ref{subsec:period collapse} we give a result on \tbf{period collapse} in the \tbf{characteristic quasi-polynomials} of the arrangements in the first and last columns of the type $B$ Shi descendant matrix.

\subsection{Type $B$ Shi  descendants}
\label{subsec:BShi-des}
\quad

We begin by recalling the definition of arrangements interpolating between $\Shi(A_{\ell-1})$ and $\Ish(A_{\ell-1})$ due to Duarte and Guedes de Oliveira \cite{DG18}.
Let $2 \leq k \leq \ell$. Define 
$$
\Hh_{\ell}^{k}  \coloneqq \Cox(A_{\ell-1}) 
\cup \Set{ x_{1}-x_{j}=i  | 1 \leq i < j \leq \ell, i < k}
\cup \Set{ x_{i}-x_{j}=1  | k \leq i < j \leq \ell}. 
$$
These arrangements  interpolate between $\Shi(A_{\ell-1}) = \Hh_{\ell}^{2} $ and  $ \Ish(A_{\ell-1}) = \Hh_{\ell}^{\ell} $ as $k$ varies. 
A notable property of these arrangements is that they share the same characteristic polynomial $t(t-\ell)^{\ell-1}$ \cite[Theorem 2.2]{DG18}. 
 
 The relationship of the arrangements $\Hh_{\ell}^{k}$ is studied more closely when they can be defined as \tbf{$\psi$-digraphic arrangements} (Definition \ref{def:A-psi-digraphic}) in a recent work of Abe and the authors \cite{ATT21}. 
 It is shown that the sequence 
$$  \Shi(A_{\ell-1})=\Hh_{\ell}^{2} \longrightarrow\Hh_{\ell}^{3}\longrightarrow\cdots\longrightarrow \Hh_{\ell}^{\ell} =  \Ish(A_{\ell-1}) $$
can be derived by certain operations on digraphs, the so-called \tbf{coking elimination}   \cite[Definition 2.15 and Theorem 3.13]{ATT21}.
Compared to the combinatorial symmetry \eqref{eq:CS}, this gives another way to understand the Ish arrangement $\Ish(A_{\ell-1}) $ as the final arrangement in the digraphic sequence above starting from $ \Shi(A_{\ell-1})$.

Furthermore, it was proved that the coking elimination operations mentioned previously preserve characteristic polynomial and freeness  \cite[Theorems 3.1 and 4.1]{ATT21}. 
In particular, the supersolvability (hence freeness) of the cone $\cc\Ish(A_{\ell-1}) $ implies the freeness of all other cones in the sequence.

 If an element $e$ appears $d\ge0$ times  in a multiset $M$, we  write $e^d \in M$.

\begin{theorem}[{\cite[Theorem 2.2]{DG18}}, {\cite[Theorem 1.6]{ATT21}}]\label{Shi/Ish interpolation}
The cones $\cc\Hh_{\ell}^{k}$ over the arrangements $\Hh_{\ell}^{k}$ $(2 \leq k \leq \ell)$ all are free with the same multiset of exponents 
$$\exp(\cc\Hh_{\ell}^{k}) = \{0,1, \ell^{\ell-1} \}.$$ 
\end{theorem}

\textbf{Notation}. For subsequent discussion, we need to fix some notation. 
 For integers $a\le b$ and $\ell\ge1$, denote $[a,b] \coloneqq \{n\in\Z\mid a \le n\le b\}$ and $[\ell] \coloneqq [1,\ell]$. 
For a subset $N\subseteq \Z$, denote $-N \coloneqq \{-n\mid n\in N\}.$ 
For notational convenience,  when writing the defining equation of a hyperplane, e.g., by $x_{i} =N$ and  $x_{i} =  Nz$ we mean the affine coordinate hyperplanes $x_{i}=n$ and its homogenizations $x_{i} =nz$ for all  $ n \in N$. 

More recently, M\"ucksch, R\"ohrle and the first author  \cite{MRT23} defined the \tbf{Shi descendants} as generalization of the arrangements $\Hh_{\ell}^{k}$ in the study of \tbf{flag-accurate arrangements}  (Definition \ref{def:FA-arr}), a subclass of free arrangements involving freeness and exponents of restrictions.

Let  $m\ge 0$, $1 \le k \le \ell$, $0 \le p \le \ell$ be integers. 
The \textbf{Shi descendant} $\A^{p,k}_\ell(m)  $ (see \cite[Definition 7.15]{MRT23} when $d=0$) is the arrangement consisting of the following hyperplanes:
\begin{align*}
x_{i} - x_{j} &= 0\quad  (1 \leq i < j \leq \ell), \\
x_{i} - x_{j} &= 1 \quad (1 \leq i < j \leq \ell+1-k), \\
x_i&= [1- m-\min\{\ell-i+1, k\}  , 0 ]  \quad (1 \leq   i \leq p), \\
x_i&= [- m-\min\{\ell-i+1, k\}   , 0]  \quad (p<  i \leq \ell).
\end{align*}

We may regard $\A^{p,k}_\ell (m)$ as the entry $a_{p,k}$ in an $(\ell+1)\times \ell$ matrix $(a_{p,k})_{0 \le p \le \ell, 1 \le k \le \ell}$. 
We call this matrix the  \tbf{Shi descendant matrix}. 
We use the same name to call the matrix consisting of the cones over the Shi descendants.

The motivation to define the Shi descendants is that each row of the matrix can also be constructed by coking elimination starting from a Shi-like arrangement \cite[Proposition 7.16]{MRT23}. 
These Shi arrangements (or $\A^{p,1}_\ell(m)$ for $0 \le p \le \ell$ in the first column of our matrix) were considered earlier by Athanasiadis \cite[Theorem 3.1]{Atha98} with a nice property that they contain sufficient deletions and restrictions  in order to apply the addition-deletion theorem  \ref{thm:AD} to guarantee their (inductive) freeness. 

The essentialization of the arrangement $\Hh_{\ell+1}^{k+1}$ for each $ 1 \leq k \leq \ell $ is affinely equivalent to $\A^{\ell,k}_\ell (1) = \A^{0,k}_\ell (0)$  \cite[Proposition 2.11]{ATT21}. 
Thus $\Hh_{\ell+1}^{k+1}$ can be found in the $\ell$-th row of the Shi descendant  matrix when $m=1$, or in the $0$-th row when $m=0$. 
In this way the Shi descendants can be viewed as ``vertical" generalization of the arrangements $\Hh_{\ell}^{k}$. 
 
 \begin{theorem}[{\cite[Theorem 7.22]{MRT23}}]
\label{thm:MRT-main1}
Let $m\ge 0$, $1 \le k \le \ell$ and $0 \le p \le \ell$. 
The cone $\cc\A^{p,k}_\ell(m)$ in the Shi descendant matrix is flag-accurate with exponents 
$$\exp(\cc\A^{p,k}_\ell(m)) = \{1, (\ell + m)^{p}, (\ell + m+1)^{\ell-p} \}.$$ 
\end{theorem}

We remark that the freeness of the Shi descendants can be shown by using the method in \cite{ATT21} which relies on the supersolvability of the Ish-like arrangements in the last column (see \cite[Remark 7.19]{MRT23}). 
In particular, the arrangements in the same row have the same multiset of exponents hence characteristic polynomial. 
The Shi/Ish duality for flag-accuracy in Theorem \ref{thm:MRT-main1} is proved by a different ``reverse'' approach by first employing the flag-accuracy of the Shi-like arrangements in the first column, then extending to all other arrangements in the matrix. 

All concepts discussed above are defined based on a root system of type $A$. 
We now introduce a type $B$ analog of the Shi  descendants.
\begin{definition}
\label{def:B-SI}
Let $m\ge 1, 1 \le k \le \ell$ and $0 \le p \le \ell$. 
The   \textbf{type $B$ Shi  descendant} $\B^{p,k}_\ell(m)$ is the arrangement  in $ \mathbb{R}^{\ell} $
consisting of the following hyperplanes
\begin{align*}
x_{i}\pm x_{j} &= 0\quad  (1 \leq i < j \leq \ell), \\
x_{i}\pm x_{j} &= 1 \quad (1 \leq i < j \leq \ell+1-k), \\
x_i&=  [2-m-\min\{\ell-i+1, k\},\,\min\{\ell-i+1, k\}+m-1]  \quad (1 \leq   i \leq p), \\
x_i&=  [1-m-\min\{\ell-i+1, k\},\,\min\{\ell-i+1, k\}+m-1] \quad (p<  i \leq \ell).
\end{align*}
\end{definition}

We will  define further concepts for type $B$. 
For simplicity, sometimes we omit the term ``type $B$" if no confusion arises. 
For example, we will call the arrangements in Definition \ref{def:B-SI} the Shi  descendants $\B^{p,k}_\ell(m)$. 

We may view each $\A^{p,k}_\ell(m)$ as a subarrangement of  $\B^{p,k}_\ell(m+1)$.
In a similar way we define the  \tbf{type $B$ Shi descendant matrix} as the matrix consisting of the  Shi  descendants $\B^{p,k}_\ell(m)$:
$$
\begin{pmatrix}
    \B^{0,1}_\ell(m)& \overset{\small{k\, \text{varies}}}{\longrightarrow} & \B^{0,\ell}_\ell(m)\\
\downarrow    \scriptsize{\text{$p$ varies}}  & \vdots  & \vdots  \\
\B^{\ell,1}_\ell(m) & \cdots & \B^{\ell,\ell}_\ell(m) 
          \end{pmatrix}.
       $$
      
  Similar to the type $A$ case, we will describe the relation between the arrangements in the same row when we introduce the  \tbf{type $B$ $\psi$-digraphic arrangement} and \tbf{coking elimination} in \S\ref{subsec:BDA} (see Theorem \ref{thm:BSI-seq}). 
  This will also justify the definition of  $\B^{p,k}_\ell(m)$ as a counterpart of $\A^{p,k}_\ell(m)$.

The arrangements $\B^{p,1}_\ell(m)$ in the first and  $\B^{p,\ell}_\ell(m)$ in the last column play a role of Shi-like and Ish-like arrangements, respectively. 
These arrangements are of our main interest, and we give their descriptions explicitly below. 
For $0 \le p \le \ell$,
\begin{equation}
\label{eq:1stcol}
  \B^{p,1}_\ell(m): 
\begin{cases}
x_{i}\pm x_{j} &= [0 , 1]\quad  (1 \leq i < j \leq \ell), \\
x_i&=  [1-m , m] \quad (1 \leq   i \leq p), \\
x_i&=[-m , m]  \quad (p<  i \leq \ell).
\end{cases}
\end{equation}

\begin{equation}
\label{eq:lastcol}
  \B^{p,\ell}_\ell(m):
  \begin{cases}
x_{i}\pm x_{j} &= 0\quad  (1 \leq i < j \leq \ell), \\
x_i&=  [1-m+i-\ell , m+\ell-i] \quad (1 \leq   i \leq p), \\
x_i&=  [-m+i-\ell , m+\ell-i]   \quad (p<  i \leq \ell).
\end{cases}
\end{equation}

In particular, take $m=1$, we obtain $\Shi(B_\ell) = \B^{\ell,1}_\ell(1)$ at bottom-left corner, and $\Ish(B_\ell) = \B^{\ell,\ell}_\ell(1)$ at bottom-right corner. 
See Figure \ref{fig:BSI-arr} for the Shi descendant  matrix for $\ell=3$.

Our first main result in the present paper is an analog of Theorem \ref{thm:MRT-main1}.
\begin{theorem}
 \label{thm:main-BSI-free-SS}
Let $m\ge 1, 1 \le k \le \ell$ and $0 \le p \le \ell$. 
The cone $\cc\B^{p,k}_\ell(m)$ in the Shi descendant matrix is flag-accurate with exponents 
$$\exp(\cc\B^{p,k}_\ell(m)) = \{1, (2m+2\ell-2)^{p}, (2m+2\ell-1)^{\ell-p}\} .$$ 
Moreover, if $\ell \ge 2$ then $\textbf{c} \B^{p,k}_\ell(m)$ is not supersolvable except the arrangements in the last column of the  matrix.
 \end{theorem}

\begin{corollary}\label{cor:Shi/Ish freeness and char poly}
The cones over $\Shi(B_{\ell})$ and $\Ish(B_{\ell})$ are free with the same multiset of exponents $\{1,(2\ell)^{\ell}\}$.
In particular, $\Shi(B_{\ell})$ and $\Ish(B_{\ell})$ have the same characteristic polynomial $(t-2\ell)^{\ell}$. 
Moreover, the cone over $\Ish(B_{\ell})$ is supersolvable. 
\end{corollary}
 
 The proof of Theorem \ref{thm:main-BSI-free-SS} is similar in spirit to Theorem \ref{thm:MRT-main1}. 
 The argument for the Ish arrangements in the last column is easiest which follows directly from the modular coatom technique \ref{prop:modular coatom}. 
 For the remaining arrangements, we first need to show the flag-accuracy of the Shi arrangements in the first column, then use the modular coatom technique to transfer this property to the other arrangements along each row in the matrix.
 
It was observed in \cite[\S1]{AR12} that the Ish arrangement $ \Ish(A_{\ell-1}) $ is sort of a ``toy model" for the Shi arrangement $ \Shi(A_{\ell-1}) $ in the sense that for any property $P$ that they share, the proof that $ \Ish(A_{\ell-1}) $ satisfies $P$ is easier than the proof that $ \Shi(A_{\ell-1}) $ satisfies $P$. 
 As the preceding paragraph suggests, this observation continues to be the case for type $B$ regarding the flag-accuracy in Theorem \ref{thm:main-BSI-free-SS}.

\subsection{Type $B$ $N$-Ish arrangements}
\label{subsec:BN-Ish}
\quad
Let $N = (N_{2}, \dots, N_{\ell})$ be a tuple of finite sets $N_{i} \subseteq \mathbb{Z}$ (not necessarily of the form $[a_i,b_i]$). 
Abe, Suyama, and the second author \cite{AST17} defined the \tbf{$N$-Ish arrangement} $\A(N)$ as a generalization of the Ish arrangement $\Ish(A_{\ell-1}) $ (or more generally, of the Ish arrangements $  \A^{p,\ell}_\ell(m)$ for $0 \le p \le \ell$ in the last column of the type $A$ Shi descendant matrix).
The $N$-Ish arrangement $\A(N)$ consists of the following hyperplanes: 
\begin{align*}
x_{i}-x_{j} &= 0 \quad (1 \leq i < j \leq \ell), \\
x_{1}-x_{i} &= N_{i} \quad (2 \leq i \leq \ell).  
\end{align*}

It is shown that freeness and supersolvability of the cone over $\A(N)$ are synonyms. 

\begin{theorem}[{\cite[Theorems 1.3 and 1.4]{AST17}}]\label{thm:AN-Ish}
The following are equivalent: 
\begin{enumerate}[(1)]
\item $N$ is a \tbf{nest}. 
Namely, there exists a permutation $w$ of $\{2, \dots, \ell\}$ such that $N_{w(i)} \subseteq N_{w(i-1)}$ for every $3 \le  i \leq \ell$. 
\item  The cone $\tbf{c}\A(N)$ is supersolvable.
\item The cone $\tbf{c}\A(N)$  is  free.
\end{enumerate}
In this case, the exponents of  $\tbf{c}\A(N)$ are given by 
$$\exp(\tbf{c}\A(N)) = \{0,1\} \cup \{| N_{w(i)}| + i-2\}_{i=2}^\ell,$$ 
where  $w$ is any permutation of $\{2, \dots, \ell\}$ such that $N_{w(i)} \subseteq N_{w(i-1)}$ for every $3 \le  i \leq \ell$.
\end{theorem}

We now introduce the notion of \tbf{type $B$ $N$-Ish arrangement} as a generalization of the Ish arrangements $  \B^{p,\ell}_\ell(m)$ for $0 \le p \le \ell$ in the last column of the type $B$ Shi descendant matrix (see \eqref{eq:lastcol}) by extending the weights of the coordinate hyperplanes.
\begin{definition}
\label{def:N-Ish-B}
Let $N = (N_1, \ldots, N_\ell)$ be an $\ell$-tuple of finite sets $N_i\subseteq\Z$. 
The  \tbf{type $B$ $N$-Ish arrangement} $\B(N)$ is the arrangement  in $ \mathbb{R}^{\ell} $
consisting of the following hyperplanes:
\begin{align*}
x_{i}\pm x_{j} &= 0\quad  (1 \leq i < j \leq \ell), \\
x_i&=  N_i \quad (1 \le  i \leq \ell).
\end{align*}
\end{definition}

Some example of the type $B$ $N$-Ish arrangements already appeared in literature, e.g., it is used by Ziegler \cite[Proposition 10]{Z89} to show that the exponents of \tbf{free multiarrangements} (see \S\ref{subsec:multi}) are in general not combinatorial.
To study the freeness and supersolvability of  $\B(N)$, we consider some special conditions on tuples.
\begin{definition}
\label{def:N-cond}
Let $N = (N_1, \ldots, N_\ell)$ be an $\ell$-tuple of finite sets $N_i\subseteq\Z$. The tuple $N$ is said to be
 \begin{enumerate}[(1)]
\item  \tbf{centered} if $0 \in N_i$ for each $1 \le  i \leq \ell$, 
\item  \tbf{uneven} if for any $1 \leq i < j \leq \ell$ either  $|N_i| \ne |N_j|$ or $|N_i| = |N_j|$ is an odd integer, otherwise it is  called \tbf{even} (i.e., there exist $1 \leq i < j \leq \ell$ such that $|N_i| = |N_j|$ is an even number), 
\item  \tbf{nonnegative} if $N_i \subseteq \Z_{\ge0}$ consists of only nonnegative integers  for each $1 \le  i \leq \ell$, 
\item  a \tbf{signed nest} if there exists a permutation $w$ of $[\ell]$ such that $\pm N_{w(i)} \subseteq N_{w(i-1)}$ (equivalently, $N_{w(i)} \cup (-N_{w(i)} )\subseteq N_{w(i-1)} $) for every $2 \le  i \leq \ell$. 
Thus $N$ is a signed nest if and only if the $2$-tuple $(N_i, N_j)$ is a signed nest  for all $1 \leq i < j \leq \ell$.
\end{enumerate}
\end{definition}

Our second main result is an analog of Theorem \ref{thm:AN-Ish}.

\begin{theorem}
 \label{thm:N-Ish-B-free-SS}
 Let $N = (N_1, \ldots, N_\ell)$ be  an $\ell$-tuple  of finite sets $N_i\subseteq \Z$ satisfying two conditions: (a)  $N$ is centered and (b) $N$ is uneven.
 The following are equivalent:
 \begin{enumerate}[(1)]
\item $N$ is a signed nest.
\item  The cone $\tbf{c}\B(N)$ is supersolvable.
\item The cone $\tbf{c}\B(N)$  is  free.
\end{enumerate}
In this case, the exponents of  $\tbf{c}\B(N)$ are given by 
$$\exp(\tbf{c}\B(N)) = \{1\} \cup \{| N_{w(i)}| + 2(i-1)\}_{i=1}^\ell,$$ 
where  $w$ is any permutation of $[\ell]$ such that $\pm N_{w(i)} \subseteq N_{w(i-1)}$ for every $2 \le  i \leq \ell$.
 \end{theorem}
 
 In particular, the arrangements $\tbf{c}\B^{p,\ell}_\ell(m)$ for $0 \le p \le \ell$ in the last column of the Shi descendant matrix given in \eqref{eq:lastcol} are  supersolvable since they are $N$-Ish arrangements whose associated tuples are signed nests. 
 
 We remark that without centeredness or unevenness the tuple $N$ being a signed nest is not enough to guarantee the freeness or supersolvability or equivalence between these properties of $\tbf{c}\B(N)$ (see Example \ref{ex:NI-no}). 
 Thus unlike type $A$, freeness and supersolvability of $\tbf{c}\B(N)$ are not synonyms. 
 A full characterization of freeness or supersolvability for a type $B$ $N$-Ish arrangement remains open to us.
 
If a tuple $N$ was supposed to be centered and nonnegative, we would give a characterization for freeness and supersolvability of $\tbf{c}\B(N)$ without the need of unevenness (see Theorem \ref{thm:N-Ish-B-free-SS-relax}). 
This characterization will play an important role in the proof of Theorem \ref{thm:SI-sym-G} in the next subsection.

 In the type $A$ case, an explicit basis for the \textbf{module of logarithmic derivations} (Definition \ref{def:free-arr}) of a free (= supersolvable) $N$-Ish arrangement is known \cite[Theorem 1.4]{AST17}. 
 The situation is more complicated in type $B$.
 We could give an explicit basis for $\tbf{c}\B(N)$ when $N_i=[-m_i,m_i]$ with $ m_1 \ge m_2 \ge \cdots \ge m_\ell \ge 0$ (Theorem \ref{thm:basis}). 
This $\tbf{c}\B(N)$ also satisfies all conditions in Theorem \ref{thm:N-Ish-B-free-SS}, and has the top-right descendant $ \tbf{c}\B^{0,\ell}_\ell(m)$ and $\Cox(B_\ell)$ as its specializations. 
It would be interesting to find a basis for other free (or just supersolvable) type $B$ $N$-Ish arrangements.

\subsection{Type $B$ deleted Shi and Ish arrangements}
\label{subsec:Bdelete}
\quad

For a loopless digraph $G=([\ell],E_G)$ on $[\ell]$ with edge set $E_G \subseteq \{(i,j) \mid 1 \leq i < j \leq \ell\}$, Athanasiadis \cite{A96, Atha98} defined the \tbf{deleted Shi arrangement} $\Shi(G)$ that interpolates between $\Cox(A_{\ell-1})$ and  $\Shi(A_{\ell-1})$, and characterized completely its freeness and supersolvability. 
Using the following digraphic version of the combinatorial symmetry  \eqref{eq:CS}
\begin{equation*}
\label{eq:digraph-CS}
\{x_{i}-x_{j}= 1\} \longleftrightarrow \{ x_{1}-x_{j} = i\} \quad ((i,j) \in E_G),
\end{equation*}
Armstrong and Rhoades  \cite{AR12} defined the \tbf{deleted Ish arrangement} $\Ish(G)$ that interpolates between  $\Cox(A_{\ell-1})$ and $\Ish(A_{\ell-1})$:
\begin{equation}
\label{eq:GSI}
\begin{aligned}
  \Shi(G) & \coloneqq \Cox(A_\ell) \cup  \Set{ x_{i}-x_{j}= 1 | (i,j) \in E_G }, \\
 \Ish(G) & \coloneqq \Cox(A_\ell) \cup \Set{ x_{1}-x_{j} = i | (i,j) \in E_G }.
\end{aligned}
\end{equation}

The following analog of Theorems \ref{Shi/Ish characteristic polynomial}, \ref{Shi/Ish freeness} and \ref{Shi/Ish chamber} for the deleted Shi and Ish arrangements holds.

\begin{theorem}[{\cite[Main Theorem]{AR12}, \cite[Theorem 6]{LRW14}, \cite[Corollary 5.3]{AST17}}]
\label{thm:GSI}
The following Shi/Ish dualities hold: 
\begin{enumerate}[(1)]
\item The arrangements $\Shi(G)$ and $\Ish(G)$ have the same characteristic polynomial. 
\item $\tbf{c} \Shi(G)$ is free $\Leftrightarrow \cc\Ish(G)$ is free $\Leftrightarrow \cc\Ish(G)$ is supersolvable. 
\item The arrangements $\Shi(G)$ and $\Ish(G)$ have the same number of regions with $c$ ceilings and $d$ degrees of freedom for any nonnegative integers $c,d$.
\end{enumerate}
\end{theorem}

Recall $\Shi(B_{\ell})$ and $\Ish(B_{\ell})$  from Definition \ref{def:Shi-Ish-B}. 
One may wonder whether an analog of Theorem \ref{thm:GSI} holds for type $B$. 
Unfortunately, it is not always the case. 
The subarrangement $\Shi(B_{3})\setminus\{x_{1}=1\}$ of $\Shi(B_{3})$ has characteristic polynomial  $t^{3}-17t^{2}+98t-191$. 
However, there is no subarrangement of $\Ish(B_{3})$ containing $\Cox(B_{3})$ such that the characteristic polynomial coincides with this polynomial. 

In order to have a valid analogy, we need to consider ``partial" deleted versions of $ \Shi(B_\ell)$ and $ \Ish(B_\ell)$. 
A natural idea is to consider an analogous digraphic version of \eqref{eq:RCS}.
In that case it is necessary to consider digraphs with loops.

In the remainder of this subsection, assume $G=([\ell],E_G,L_G)$ is a digraph on $[\ell]$ with loop set $L_G\subseteq[\ell]$ and edge set $E_G \subseteq \{(i,j) \mid 1 \leq i < j \leq \ell\}$. 
\begin{definition}
\label{def:Del-SI}
Define arrangements $\mcS(G)$
 and $\I(G)$ in $\R^\ell$ by
\begin{align*}
\mcS(G) & \coloneqq \Cox(B_\ell)  \cup  \Set{ x_{i} = 1 | i \in L_G} \cup  \Set{ x_{i}-x_{j}= 1 | (i,j) \in E_G}, \\
\I(G) & \coloneqq \Cox(B_\ell)  \cup  \Set{ x_{i} = 1 | i \in L_G} \cup  \Set{ x_{i} = \ell+2-j | (i,j) \in E_G}.
\end{align*}
\end{definition}

The arrangements above are ``partial" deleted versions of $ \Shi(B_\ell)$ and $ \Ish(B_\ell)$ in the sense that 
$\mcS(G)$ (resp., $ \I(G) $) interpolates between $\Cox(B_\ell)$ and  $\Shi^-(B_\ell)$ (resp., $\Ish^-(B_\ell)$) defined as below
\begin{align*}
 & \Cox(B_\ell) \subseteq  \mcS(G)   \subseteq \Shi^-(B_\ell) \coloneqq \Shi(B_\ell) \setminus \Set{ x_{i}+ x_{j}= 1 | 1 \leq i < j \leq \ell}, \\
 & \Cox(B_\ell) \subseteq \I(G)   \subseteq \Ish^-(B_\ell) \coloneqq \Ish(B_\ell)  \setminus \Set{ x_{i} =  -(\ell+1-j )| 1 \leq i < j \leq \ell}.
\end{align*}

Our third main result is an analog of Theorem \ref{thm:GSI}(2). 
\begin{theorem}
 \label{thm:SI-sym-G}
The following Shi/Ish duality between the cones over $\mcS(G)$ and $\I(G)$ holds:
$$\tbf{c} \mcS(G) \text{ is free } \Leftrightarrow \tbf{c}\mcS(G) \text{ is supersolvable } \Leftrightarrow \tbf{c} \I(G) \text{ is free } \Leftrightarrow \tbf{c}\I(G) \text{ is supersolvable}.$$
Furthermore, any of the above conditions occurs if and only if $G$ has one of the following forms: 
 \begin{enumerate}[(a)]
\item $|E_G|\ge1$ and  all the edges in $G$ have the same initial vertex with possible loop at this vertex, and there are no loops at any other vertices, 
\item $|E_G|\ge2$ and all the edges in $G$ have the same terminal vertex, and there are no loops at any vertices (including  the terminal vertex),
\item $G$ has no edges but loops at some vertices. 
\end{enumerate}
In this case, the exponents of  $\tbf{c} \mcS(G)$ and $\tbf{c}\I(G)$ are both given by 
$$\exp(\tbf{c} \mcS(G)) = \exp(\tbf{c}\I(G)) = \{|E_G|+|L_G|+1\} \cup \{2i-1\}_{i=1}^\ell.$$
 \end{theorem} 
 
 In particular, the deleted arrangements $\mcS(G)$ and $\I(G)$ have the same characteristic polynomial when their associated cones are free. 
 Our calculation so far suggests the following analog of Theorem \ref{thm:GSI}(1). 
 \begin{conjecture}
\label{conj:same-CP}
 $\mcS(G)$ and $\I(G)$ have the same characteristic polynomial for any digraph $G$.
 \end{conjecture}
 
 Notice that $\Shi(B_{\ell})$ and $\Ish(B_{\ell})$ are the first and last arrangements in the last row of the Shi descendant matrix. 
 It would also be interesting to define similar  ``partial" deleted versions of the first and last arrangements in the other rows, and study the Shi/Ish dualities similar to  Theorem \ref{thm:GSI}.

\subsection{Characteristic quasi-polynomials and period collapse}
\label{subsec:period collapse}
\quad
Given an \tbf{integral} arrangement, Kamiya, Takemura and Terao \cite{KTT08, KTT11} introduced the notion of \tbf{characteristic quasi-polynomial}, which enumerates the cardinality of the complement of the arrangement modulo a positive integer (Theorem \ref{thm:KTT}). 
The most popular candidate for \tbf{periods} of these quasi-polynomials is the \tbf{lcm period} (Definition \ref{def:lcm}). 

The lcm period is known to be the \tbf{minimum period} when the arrangement is central \cite[Theorem 1.2]{HTY22}. 
The minimum period in the noncentral case remains unknown. 
We say that \tbf{period collapse} occurs in the characteristic quasi-polynomial of a noncentral arrangement when the minimum period is strictly less than the lcm period.

The first example of period collapse arises from the \tbf{extended Shi arrangement} of a \tbf{root system} of an arbitrary type but type $A$ \cite[Theorem 5.1]{Y18W}. 
Note that in type $A$, the lcm period is equal to $1$ (see e.g., \cite[Corollary 3.2]{KTT10}) hence no period collapse occurs. 
In particular, the characteristic polynomial and quasi-polynomial of $\Shi(A_{\ell-1}) $ or $\Ish(A_{\ell-1}) $ coincide.  

Higashitani, Yoshinaga and the first author showed that period collapse occurs in any dimension $\ge 1$, occurs for any lcm period $\ge 2$, and the minimum period when it is not the lcm period can be any proper divisor of the lcm period \cite[Theorem 1.2]{HTY22}. 
Despite the fact that any sort of period collapse is possible, what makes period collapse happen is still an interesting question.

Our fourth (and final) main result is a new example of period collapse.
\begin{theorem}
 \label{thm:1stcol-PC-intro}
The characteristic quasi-polynomial of each arrangement $\B^{p,1}_\ell(m)$ for $0 \le p \le \ell$ in the first column of the Shi descendant matrix given in \eqref{eq:1stcol} is actually a polynomial. 
Hence period collapse occurs in these quasi-polynomials. 
 \end{theorem} 

One may wonder if a Shi/Ish duality regarding period collapse holds for $\Shi(B_{\ell})$ and $\Ish(B_{\ell})$. 
Unfortunately, it is not the case. 
It can be readily verified that both lcm and minimum periods for $\Ish(B_{2})$ are equal to $2$ (see also \S\ref{sec:CQP-PC} for further details).

\section{Preliminaries}
\label{sec:free-arr}

\subsection{Free arrangements}
\label{subsec:free}
\quad
We begin by recalling basic concepts and preliminary results on free arrangements. Our standard reference is \cite{OT92}.
Let $ \mathbb{K} $ be a field and let $V=  \mathbb{K}^{\ell} $. 
A \textbf{hyperplane} in $V$ is an affine subspace of codimension $1$ of $V$.
An \textbf{arrangement} is a finite collection of hyperplanes in  $ V $. 
An arrangement is called \textbf{central} if every hyperplane in it passes through the origin. 

Let  $ \mathcal{A} $ be an arrangement.
Define the \textbf{intersection poset} $ L(\mathcal{A}) $ of $ \mathcal{A} $ by 
\begin{align*}
L(\mathcal{A})  \coloneqq \Set{\bigcap_{H \in \mathcal{B}}H \neq \varnothing |  \mathcal{B} \subseteq \mathcal{A} },
\end{align*}
where the partial order is given by reverse inclusion $X\le Y\Leftrightarrow Y\subseteq X$ for $X, Y \in L(\A)$. 
We agree that $ V $ is a unique minimal element in $ L(\mathcal{A}) $ as the intersection over the empty set. 
Thus $ L(\mathcal{A}) $ is a semi-lattice which can be equipped with the rank function $ r(X)  \coloneqq \operatorname{codim}(X) $ for $X \in L(\mathcal{A})$. 
We also define the \textbf{rank} $r(\A)$ of $\A$ as the rank of a maximal element of $L(\A)$.
The intersection poset $ L(\mathcal{A}) $ is often referred to as the  \textbf{combinatorics} of  $ \mathcal{A} $. 

The \textbf{characteristic polynomial} $ \chi_{\mathcal{A}}(t) \in \mathbb{Z}[t] $ of $ \mathcal{A} $ is defined by
\begin{align*}
\chi_{\mathcal{A}}(t)  \coloneqq \sum_{X \in L(\mathcal{A})}\mu(X)t^{\dim X}, 
\end{align*}
where $ \mu $ denotes the \textbf{M\"{o}bius function} $ \mu \colon L(\mathcal{A}) \to \mathbb{Z} $ defined recursively by 
\begin{align*}
\mu\left( V \right)  \coloneqq 1 
\quad \text{ and } \quad 
\mu(X)  \coloneqq -\sum_{\substack{Y \in L(\mathcal{A}) \\ X \subsetneq Y  \subseteq V}}\mu(Y). 
\end{align*}

Let $ \{x_{1}, \dots, x_{\ell}\} $ be a basis for the dual space $V^{\ast} $ and let $ S  \coloneqq \mathbb{K}[x_{1}, \dots, x_{\ell}] $. 
The \textbf{defining polynomial} $Q(\A)$ of $\A$ is given by
$$Q(\A) \coloneqq \prod_{H \in \A} \alpha_H \in S,$$
where $ \alpha_H=a_1x_1+\cdots+a_\ell x_\ell+d$  $(a_i, d \in \mathbb{K})$ satisfies $H = \ker \alpha_H$.

The \textbf{cone} $ \textbf{c}\A$ over $\A$ is the central arrangement in $\mathbb{K}^{\ell+1}$ with the defining polynomial
$$Q(\textbf{c}\A) \coloneqq z\prod_{H \in \A} {}^h\alpha_H \in \mathbb{K}[x_1,\ldots, x_\ell,z],$$
where $ {}^h\alpha_H \coloneqq a_1x_1+\cdots+a_\ell x_\ell+dz$ is the homogenization of $\alpha_H$, and $z=0$ is the \tbf{hyperplane at infinity}, denoted $H_{\infty}$. 

A $\K$-linear map $\theta:S\to S$ which satisfies $\theta(fg) = \theta(f)g + f\theta(g)$ is called a \tbf{derivation}.
Let $\Der(S)$ be the set of all derivations of $S$. 
It is a free $S$-module with a basis $\{\partial/\partial x_1,\ldots,\partial/\partial x_{\ell}\}$ consisting of the usual partial derivatives.
We say that a nonzero derivation $\theta  = \sum_{i=1}^\ell f_i \partial/\partial x_{i}$  is \tbf{homogeneous of degree} $p$ if each nonzero coefficient $f_i$ is a homogeneous polynomial of degree $p$ \cite[Definition 4.2]{OT92}.
In this case we write $\deg{\theta} = p$.

The concept of free arrangements was defined by Terao for central arrangements \cite{T80}. 
In the remainder of this section, unless otherwise stated, assume $\A$ is a central arrangement in  $V=\K^\ell$. 
\begin{definition}[Free arrangement {\cite[Definitions 4.5 and 4.15]{OT92}}]\label{def:free-arr}
The \textbf{module $D(\mathcal{A}) $ of logarithmic derivations}  is defined by 
	\begin{equation*}
	D(\A)  \coloneqq  \{ \theta \in \Der(S) \mid \theta(Q(\A)) \in Q(\A)S\}.
	\end{equation*}
	We say that $\A$ is \tbf{free} if the module $D(\A)$ is a free $S$-module. 
\end{definition}

If $\A$ is a free arrangement, we may choose a homogeneous basis $\{\theta_1, \ldots, \theta_\ell\}$ for $D(\A)$.
Then the degrees of the $\theta_i$'s are called the \tbf{exponents} of $\A$ \cite[Definition 4.25]{OT92}. 
They are uniquely determined by $\A$. 
In that case we write 
\[
\exp(\A)  \coloneqq  \{\deg{\theta_1},\ldots,\deg{\theta_\ell}\}\]
for the multiset of exponents of $\A$. 
If $\exp(\A) = \{d_1, \ldots, d_\ell\}$ with $ d_1 \leq \cdots \leq d_\ell$, we write  $\exp(\A) = \{d_1,\ldots, d_\ell\}_\le$.

Though the freeness was defined in an algebraic sense, it is related to the combinatorics of arrangements due to a remarkable result of Terao. 

\begin{theorem}[Factorization theorem {\cite[Main Theorem]{T81}}, {\cite[Theorem 4.137]{OT92}}]\label{thm:Factorization}
If $\A$ is free with $\exp(\A) =\{d_{1}, \dots, d_{\ell}\} $, then 
$$\chi_\A(t)= \prod_{i=1}^\ell (t-d_i).$$
\end{theorem}

For each $X \in L(\A)$, define the \textbf{localization} of $\A$ on $X$  by 
$${\A}_X  \coloneqq \{ K \in {\A} \mid X \subseteq K\} \subseteq \A,$$
and the \textbf{restriction} ${\A}^{X}$ of ${\A}$ to $X$ by 
$${\A}^{X}  \coloneqq \{ K \cap X \mid K \in{\A }\setminus {\A}_X\}.$$

Fix $H \in \A$, denote $\A' \coloneqq \A\setminus \{H\}$ and $\A'' \coloneqq \A^H$. 
We call $(\A, \A', \A'')$ the triple with respect to the hyperplane $H \in\A$. 

 \begin{theorem}[Deletion-restriction formula  {\cite[Theorem 2.56]{OT92}}]
 \label{thm:DR}
 If $(\A, \A', \A'')$ is a triple of arrangements, then
 \begin{equation*}
\chi_{\A}(t)=
\chi_{\A'}(t)-
\chi_{\A''}(t). 
\end{equation*}
\end{theorem}

 \begin{theorem}[Addition-deletion theorem {\cite{T80}},  {\cite[Theorems 4.46 and 4.51]{OT92}}]
 \label{thm:AD}
Let $\A$ be a nonempty arrangement and let $H \in \A$. Then two of the following imply the third:
\begin{enumerate}[(1)]
\item $\A$ is free with $\exp(\A) = \{d_1, \ldots, d_{\ell-1}, d_\ell\}$.
\item  $\A'$ is free with $\exp(\A')=\{d_1, \ldots, d_{\ell-1}, d_\ell-1\}$. 
\item $\A''$ is free with $\exp(\A'') = \{d_1, \ldots, d_{\ell-1}\}$.
\end{enumerate}
Moreover, all the three hold true if $\A$ and $\A'$ are both free.
\end{theorem}
 
We recall an improvement of the addition part of the theorem above ($(2)+(3) \Rightarrow (1)$) due to Abe.

 \begin{theorem}[Division theorem {\cite[Theorem 1.1]{Abe16}}]
 \label{thm:DT}
Assume that there is a hyperplane $H \in\A$ such that $\chi_{\A''}(t)$ divides $\chi_{\A}(t)$  and that $\A''$ is free. Then $\A$ is free.
\end{theorem}

Let $\emptyset_{\ell} $ denote the $ \ell $-dimensional \textbf{empty arrangement}, that is, the arrangement in $ V $ consisting of no hyperplanes. 
An arrangement $ \mathcal{A} $ in $ V $ is called \textbf{essential} if $ r(\mathcal{A}) = \ell $. 
Any arrangement $ \mathcal{A} $ of rank $r$ in $V$ can be written as the \tbf{product} (see e.g., \cite[Definition 2.13]{OT92})  $ \mathcal{A}^{\mathrm{ess}} \times\emptyset_{\ell-r}$, where  $\mathcal{A}^{\mathrm{ess}}$ is the \textbf{essentialization} (see e.g., \cite[\S1.1]{St07}) of $ \mathcal{A} $ which is an essential arrangement   of rank $r$. 
Moreover, $ \mathcal{A} $ is free if and only if $ \mathcal{A}^{\mathrm{ess}} $ is free \cite[Proposition 4.28]{OT92}. 
In this case, $ \exp (\A ) = \exp (\A^{\mathrm{ess}} ) \cup \{ 0^{\ell-r} \} $. 

A special subclass of free arrangements is recently studied by M\"ucksch, R\"ohrle and the first author.

\begin{definition}[Flag-accurate arrangement {\cite[Definition 1.1]{MRT23}}]
\label{def:FA-arr}
A free arrangement $\A$ of rank $r$ with $\exp(\A) = \{0^{\ell-r}, d_1, \ldots, d_{r}\}_\leq$ is called \tbf{flag-accurate} if there exists a flag of subspaces in $L(\A)$
$$X_1 \subseteq X_2  \subseteq \cdots  \subseteq X_{\ell-1} \subseteq X_r=V$$
such that $\dim(X_i)=\ell-r+i$ and $\A^{X_i}$ is free with  $\exp(\A^{X_i}) = \{0^{\ell-r}, d_1, \ldots, d_i\}_ \leq$ for  each $1 \le i \le r$.
\end{definition}

The flag-accurate arrangements form a subclass of both \tbf{accurate} and \tbf{divisionally free} arrangements, the concepts due to M\"ucksch and R\"ohrle \cite{MR21}, and Abe \cite{Abe16}, respectively.
In particular,  flag-accuracy is a combinatorial property  \cite[Remark 1.3]{MRT23}.

\subsection{Supersolvable arrangements}
\label{subsec:SS}
\quad
Now we recall the definition of \textbf{supersolvable} arrangements following e.g., \cite[\S 2]{ER94}. 
 \begin{definition}
 \label{def:modco}
Given a subarrangement $\B\subseteq \A$ of a central arrangement $\A$, we say $\B$  is  a \textbf{modular coatom} of $\A$ if 
\begin{enumerate}[(1)]
\item $r(\A) = r(\B)+1$, and
\item for any distinct $ H, H^{\prime} \in \mathcal{A}\setminus \B$, there exists $ H^{\prime\prime} \in \B$ such that $ H\cap H^{\prime} \subseteq H^{\prime\prime} $. 
\end{enumerate}
\end{definition}

  \begin{proposition}[{\cite[Theorem 4.3]{BEZ90}}]
 \label{prop:BEZ}
A subarrangement $\B\subseteq \A$ is a modular coatom if and only if $\B = \A_X$ for some coatom $X \in L(\A)$ such that $ X+Y \in L(\A) $ for all $ Y \in L(\A) $. 
 \end{proposition}

 \begin{definition}[Supersolvable arrangement]
 \label{def:supersolvable}
A central arrangement $\A$ of rank $r$ is called \textbf{supersolvable} if there exists a chain of arrangements, called an \tbf{M-chain},
$$\varnothing  =  \mathcal{A}_{0} \subseteq \mathcal{A}_{1} \subseteq \dots \subseteq \mathcal{A}_{r} = \A,$$
in which $ \mathcal{A}_{{i}} $ is a modular coatom of $ \mathcal{A}_{{i+1}} $ for each $0 \le i \le r-1$. 
\end{definition}

Next we describe a relationship between supersolvable and free arrangements.

 \begin{theorem}[{\cite[Theorem 4.2]{JT84}}]
 \label{thm:exp-ss}
If $\A$ is supersolvable, then $\A$ is free. 
Furthermore, if $\A$ has an \textrm{M}-chain $\varnothing  =  \mathcal{A}_{0} \subseteq \mathcal{A}_{1} \subseteq \dots \subseteq \mathcal{A}_{r} = \A,$ then
$\exp(\A) = \{0^{\ell-r(\A)}, d_1, \ldots, d_{r(\A)}\}$ where $d_i  \coloneqq  | \mathcal{A}_{{i}} \setminus  \mathcal{A}_{{i-1}}|$.
 \end{theorem}
 
 The theorem below says supersolvability and freeness are closed under taking localization. 
 
  \begin{theorem}[{\cite[Proposition 3.2]{St72}} and  {\cite[Theorem 4.37]{OT92}}]
 \label{thm:localization-ss-free}
 If $\A$ is  supersolvable (resp., free), then the localization $\A_X$ is  supersolvable (resp., free) for  any $X \in L(\A)$.
 \end{theorem}

The following well-known property of modular coatoms (see e.g., \cite[Proposition 3.3]{MRT23}) will be used often in the proofs of our results to come.
 \begin{proposition}[Modular coatom technique] 
 \label{prop:modular coatom}
 Let $\A$ be a central arrangement and let  $\B\subseteq \A$ be a modular coatom of $\A$. 
The following statements hold.
\begin{enumerate}[(1)]
\item $ \mathcal{A} $ is supersolvable (resp., free) if and only if $ \B $ is supersolvable (resp., free).
In this case, $ \exp (\A^{\mathrm{ess}}) = \exp (\B^{\mathrm{ess}} ) \cup \{|\mathcal{A}\setminus\B|\} $. 
\item  If  $\B$ is flag-accurate whose exponents do not exceed  $|\A\setminus\B|$, then $ \A $ is flag-accurate.
\end{enumerate}
\end{proposition}

Two (central) hyperplane arrangements $\A$ and $\A'$ in $V$ are said to be \textbf{(linearly) affinely equivalent} if there is an invertible (linear) affine endomorphism $\varphi: V \to V$ such that $\A'=\varphi(\A)=\{\varphi(H)\mid H\in \A\}$. 
In particular, the intersection posets of two affinely equivalent arrangements are isomorphic. 
One can prove that the supersolvability, flag-accuracy and freeness all are preserved under linear equivalence. 
In the rest of the paper, we will often identify affinely equivalent arrangements. 
Note that for  two  affinely equivalent noncentral arrangements $\A$ and $\A'$, the cones $\cc\A$ and $\cc\A'$ are linearly equivalent. 

\subsection{Multiarrangements}
\label{subsec:multi}
\quad
A \tbf{multiarrangement} is a pair $(\A, m)$ where $\A$ is a central arrangement in $V=\K^\ell$ and $m$ is a map $m : \A \to \Z_{\ge0}$, called \tbf{multiplicity}. 
Let $(\A, m)$ be a multiarrangement.
The defining polynomial $Q(\A,m)$ of $(\A, m)$ is given by 
$$Q(\A,m) \coloneqq  \prod_{H \in \A} \alpha^{m(H)}_H \in S= \mathbb{K}[x_{1}, \dots, x_{\ell}].$$
When $m(H) = 1$ for every $H \in \A$, $(\A,m)$ is simply a hyperplane arrangement. 
The \tbf{module $D(\A,m)$ of logarithmic derivations} of  $(\A, m)$ is defined by
$$D(\A,m) \coloneqq   \{ \theta\in \Der(S) \mid \theta(\alpha_H) \in \alpha^{m(H)}_HS \mbox{ for all } H \in \A\}.$$

We say that $(\A, m)$  is \tbf{free}  with the multiset $ \exp(\A, m) = \{d_{1}, \dots, d_{\ell}\} $ of \textbf{exponents}  if $D(\A,m)$ is a free $S$-module with a homogeneous basis $ \{\theta_{1}, \dots, \theta_{\ell}\}$  such that $ \deg \theta_{i} = d_{i} $ for each $ i $. 
It is known that  $(\A, m)$ is always free for $\ell\le 2$ \cite[Corollary 7]{Z89}.

For $\theta_1, \ldots , \theta_\ell \in D(\A,m)$, we define the $(\ell \times\ell)$-matrix $M(\theta_1, \ldots , \theta_\ell)$ as the matrix with $(i,j)$-th entry $\theta_j(x_i)$. 
\begin{theorem}[{\cite[Theorem 4.19]{OT92}, \cite[Theorem 8]{Z89}}]
\label{thm:criterion}
Let $\theta_1, \ldots , \theta_\ell \in D(\A,m)$. Then $\{\theta_1, \ldots , \theta_\ell\}$ forms a basis for $D(\A,m)$ if and only if
$$\det M(\theta_1, \ldots , \theta_\ell) \in \K^* \cdot Q(\A,m).$$
\end{theorem}

Let $H \in \A$. The \tbf{Ziegler restriction} $(\A^H,m^H)$ of $\A$ onto $H$ is a multiarrangement defined by 
$$m^H(X) \coloneqq |\A_X|-1 \quad \mbox{for } X \in \A^H.$$ 

We will need the following characterization for freeness of a simple arrangement in dimension $3$. 
\begin{theorem}[{\cite[Corollary 3.3]{Yo05}}]
 \label{thm:Yo-3arr}
A central arrangement $ \mathcal{A}$ in $\mathbb{K}^3$ is free if and only if 
$$ \chi_{\mathcal{A}}(t) = (t-1)(t-d_2)(t-d_3),$$
where $\exp (\A^H, m^H)=\{d_2, d_3\}$ with $H \in \A$.
\end{theorem}

\subsection{Characteristic quasi-polynomials}
\label{subsec:multi}
\quad
A function $\varphi: \Z \to \C$ is called a \tbf{quasi-polynomial} if there exist a positive integer $\rho\in\Z_{>0}$ and polynomials $f^k(t)\in\mathbb{Q}[t]$ ($1 \le k \le \rho$) such that for any $q\in\Z_{>0}$ with  $q\equiv k\bmod \rho$, 
\begin{equation*}
\varphi(q) =f^k(q).
\end{equation*}

The number $\rho$ is called a \tbf{period}, and the polynomial $f^k(t)$ is called the \tbf{$k$-constituent} of the quasi-polynomial $\varphi$. 
The smallest such $\rho$ is called the \tbf{minimum period} of  $\varphi$. 
The minimum period is necessarily a divisor of any period.

Let $\ell,n\in\Z_{>0}$ be positive integers. 
Denote by $\Mat_{\ell\times n}(\Z)$ the set of all $\ell\times n$ matrices with integer entries. 
Let $C=(c_1,\ldots,c_n) \in \Mat_{\ell\times n}(\Z)$ with nonzero column and let $b=(b_1,\ldots,b_n) \in\Z^n$. 
Set $A  \coloneqq  \begin{pmatrix} C\\ b\end{pmatrix}  \in \Mat_{(\ell + 1)\times n}(\Z)$. 
The matrix $A$ defines the following hyperplane arrangement, called \tbf{integral} arrangement in $\R^\ell$
$$\A=\A(A)   \coloneqq  \{ H_j : 1\le j\le n\},$$
where
$$H_j = H_{c_j}   \coloneqq  \{x\in \R^\ell\mid xc_j=b_j\}.$$

Let $q\in\Z_{>0}$ and $\Z_q  \coloneqq \Z/q\Z$. 
 For $a \in \Z$, let $\overline{a}  \coloneqq  a + q\Z \in \Z_q$ denote the $q$-reduction of $a$. 
For a matrix or vector $A'$ with integral entries, denote by $\overline{A'}$ the entry-wise $q$-reduction of $A'$.

The \tbf{$q$-reduction} $\A_q$ of $\A$ is defined by 
$$\A_q =\A_q (A) \coloneqq  \{ H_{j,q} \mid 1\le j\le n\},$$
where
$$H_{j,q}  \coloneqq  \{ z \in \Z_q^\ell \mid z\overline{c_j} =\overline{b_j}\}.$$

Denote $\Z_q^{\times} \coloneqq \Z_q \setminus \{0\}$.
The \tbf{complement} $\M(\A_q)$ of $\A_q$ is defined by 
$$\M(\A_q)  \coloneqq  \Z_q^\ell \setminus\bigcup_{j=1}^n H_{j,q} = \{z\in \Z_q^\ell \mid z\overline{C} -\overline{b} \in (\Z_q^{\times})^n\}.$$

For $\emptyset \ne J \subseteq [n]$, let $C_J \in \Mat_{\ell\times \#J}(\Z)$ denote the submatrix of $C$ consisting of the columns indexed by $J$.
Set $r(J) \coloneqq \rk C_J$.
Let $0<e_{J,1} \mid e_{J,2} \mid \cdots \mid e_{J,r(J)}$ be the elementary divisors of $C_J$.

\begin{definition}
\label{def:lcm}
Define the \tbf{lcm period} of $C$ by
$$\rho_C \coloneqq  \lcm \{e_{J,r(J)} \mid \emptyset \ne J \subseteq [n]\}.$$
\end{definition}

\begin{theorem}[{\cite[Theorem 2.4]{KTT08}, \cite[Theorem 3.1]{KTT11}}]
\label{thm:KTT}
There exists a monic quasi-polynomial $\chi^{\quasi}_{\A}(q)$ of degree $\ell$ with a period $\rho_{C}$ such that  for sufficiently large $q$,
$$|\M(\A_q) | = \chi^{\quasi}_{\A}(q).$$ 
This quasi-polynomial is called the \tbf{characteristic quasi-polynomial} of $\A$. 
\end{theorem}

The name ``characteristic quasi-polynomial" is made by inspiration of the following fact.

\begin{theorem}[{e.g., \cite[Remark 3.3]{KTT11}}]
\label{thm:KTT11} 
The $1$-constituent $f^1_{\A}(t)$ of $\chi^{\quasi}_{\A}(q)$ coincides with the characteristic polynomial $\chi_\A( t)$ of $\A$: 	
$$f^1_{\A}(t)=\chi_{\A}(t).$$
\end{theorem}


\section{Type $B$ Shi descendants}
\label{sec:F-SS-BSI}
\quad
We first fix some notation throughout this section. 
By $ G = (V_{G}, E_{G}) $ we mean a digraph  with vertex set $ V_{G} = [\ell]$ and edge set $E_G \subseteq \{(i,j) \mid 1 \leq i < j \leq \ell\}$. 
A directed edge $(i,j) \in E_{G}$ is considered to be \tbf{directed from $i$ to $j$}.
A  \tbf{vertex-weighted digraph} is a pair  $(G,\psi)$ where $G$ is a digraph on $ [\ell]$ and a map $ \psi \colon [\ell] \to 2^{\mathbb{Z}} $, called a \tbf{weight} on $G$. 
A weight   $ \psi$ is called an \tbf{interval weight} if each $\psi(i)$  is an integral interval, i.e., $ \psi(i) = [a_{i}, b_{i}]\subseteq \mathbb{Z}$ where $a_{i}\le b_{i}$ are integers for every $ i \in [\ell] $.

\begin{definition}[{\cite[Definition 2.1]{ATT21}}]
\label{def:A-psi-digraphic}
Let $(G,\psi)$ be a vertex-weighted digraph.
The \tbf{(type $A$) $ \psi $-digraphic arrangement} $ \A(G,\psi) $ in $ \mathbb{R}^{\ell} $ is defined by 
$$
\A(G,\psi)  \coloneqq 
\Cox(A_\ell) 
\cup \{ x_{i}-x_{j}=1  \mid (i,j) \in E_{G}\}
\cup \{  x_{i}=\psi(i) \mid 1 \leq i   \leq \ell\}.
$$
\end{definition}

We will be interested in following digraphs.
\begin{definition}
\label{def:digraphs}
The \tbf{transitive tournament} $  T_{[\ell]}$ and \tbf{edgeless digraph} $ \overline{K}_{[\ell]}$ on $ [\ell]$ are defined by 
\begin{align*}
E_{T_{[\ell]}} & \coloneqq  \{(i,j) \mid 1 \leq i < j \leq \ell\},\\
E_{\overline{K}_{[\ell]}} & \coloneqq  \varnothing.
\end{align*}
For simplicity we often use the notation $T_{\ell}$ and $ \overline{K}_{\ell}$ for $  T_{[\ell]} $ and $ \overline{K}_{[\ell]}$, respectively. 

For $1 \leq   k \leq \ell$, define the digraph  $ T_{\ell}^{k}$  on $ [\ell] $ by 
\begin{align*} 
E_{T_{\ell}^{k}} &  \coloneqq \{(i,j) \mid 1 \leq i < j \leq \ell - k + 1\}.
\end{align*}
\end{definition}

\subsection{Type $B$  $ \psi $-digraphic arrangements}
\label{subsec:BDA}
\quad
In this subsection we introduce an analog of $ \psi $-digraphic arrangements for a type $B$ root system. 
Let $(G,\psi)$ be a vertex-weighted digraph on $ [\ell] $. 
\begin{definition}
The  \textbf{type $B$  $ \psi $-digraphic arrangement} $ \B(G,\psi) $ is the arrangement  in $ \mathbb{R}^{\ell} $
consisting of the following hyperplanes:
\begin{align*}
x_{i}\pm x_{j} &= 0\quad  (1 \leq i < j \leq \ell), \\
x_{i}\pm x_{j} &= 1\quad  ((i,j) \in E_{G}, i<j), \\
x_{i}\pm x_{j} &= -1\quad  ((j,i) \in E_{G}, i<j), \\
x_i&=  \psi(i)\quad (1 \le  i \leq \ell).
\end{align*}
\end{definition}

In comparison with the type $A$ case (Definition \ref{def:A-psi-digraphic}), in type $B$ we associate to each directed edge two hyperplanes instead of one and that $\A(G,\psi) \subsetneq \B(G,\psi)$.

\begin{example}
\label{ex:BSI-pda}
The type $B$ Shi descendants (Definition \ref{def:B-SI}) are type $B$ $ \psi $-digraphic arrangements:
 $$\B^{p,k}_\ell(m)= \mathcal{B}(T_{\ell}^{k}, \psi_{\ell}^{p,k}),$$
where $ T_{\ell}^{k}$ is the digraph defined in Definition \ref{def:digraphs} and the map $ \psi_{\ell}^{p,k} \colon [\ell] \to 2^{\mathbb{Z}} $ is given by 
$$
 \psi_{\ell}^{p,k}(i)   \coloneqq \begin{cases}
 [2-m-\min\{\ell-i+1, k\},\,\min\{\ell-i+1, k\}+m-1] & (1 \leq   i \leq p), \\
 [1-m-\min\{\ell-i+1, k\},\,\min\{\ell-i+1, k\}+m-1] & (p<  i \leq \ell). 
\end{cases}
$$
\end{example}

\begin{example}
\label{ex:BN-Ish}
The type $B$ $N$-Ish arrangement from Definition \ref{def:N-Ish-B} is a type $B$ $ \psi $-digraphic arrangement:
$\B(N) = \B(G,\psi)$ where $G =  \overline{K}_{\ell}$ is the edgeless digraph on $ [\ell]$ and $\psi(i) =N_i$ for each $i$.
\end{example}

In the type $B$ case, we need the following variant of \tbf{simplicial vertex} in a vertex-weighted digraph from \cite[Definition 3.10]{ATT21}. 

\begin{definition}
\label{def:B-simplicial}
Let $v$ be a vertex in $G$ and let $\B(v) \subseteq \mathcal{B}(G,\psi)$ be the subarrangement of $\mathcal{B}(G,\psi)$ consisting of the following hyperplanes:
\begin{align*}
x_{i}\pm x_{j} &= 0 \qquad (i,j\in [\ell]\setminus\{v\}), \\
x_{i}\pm x_{j} &= 1\qquad ((i,j) \in E_{G},i<j,  i,j\in [\ell]\setminus\{v\}), \\
x_{i}\pm x_{j} &= -1 \qquad ((j,i) \in E_{G},i<j,  i,j\in [\ell]\setminus\{v\}), \\
 x_{i} &= \psi(i) \qquad (c \in \psi(i), i\in [\ell]\setminus\{v\}). 
\end{align*}
The vertex $ v $ is called \textbf{B-simplicial} in $ (G,\psi) $ if $\cc\B(v)$ is a modular coatom of $ \mathbf{c}\mathcal{B}(G,\psi) $.
\end{definition}

Let  $ G\setminus v $ denote the subgraph obtained from $G$ by removing $ v $ and the edges incident from or on $v$.  
Thus 
$$\cc\B(v)=\mathbf{c}\B\left(G\setminus v, \psi|_{[\ell]\setminus\{v\}}\right) .$$ 

\begin{proposition}
 \label{prop:isolated}
 Let  $ v $ be an isolated vertex in $(G,\psi)$, i.e., there are no edges of $G$ incident from or on $v$. 
If $0 \in \psi(i) $ for all $ i \in [\ell] $ and $ \pm\psi(v) \subseteq \psi(i) $ for all $ i\in [\ell]\setminus\{v\}$, then $ v $ is B-simplicial in $ (G,\psi) $. 
\end{proposition}
\begin{proof}
We need to check Conditions \ref{def:modco}(1) and (2). 
Condition (1) is clear. 
We show Condition (2) for some cases, the remaining ones are treated similarly. 
If $H=\{x_i +x_v =0\}$ and $H'=\{x_i -x_v =0\}$ for $i\ne v$, then we can choose $H'' = \{x_i =0\} \in \cc\B(v)$ (since  $0 \in \psi(i) $) and $ H\cap H'\subseteq H''$.  
If $H=\{x_i +x_v =0\}$ and $H'=\{x_v =a\}$ for $i\ne v$ and $a \in \psi(v)$, then we can choose $H'' = \{x_i =-a\} \in \cc\B(v)$ (since  $-\psi(v) \subseteq \psi(i)  $) and $ H\cap H'\subseteq H''$.
\end{proof}

The modular coatom technique \ref{prop:modular coatom} can be translated into digraphical terms as follows.
\begin{proposition}
\label{prop:B-simplicial}
Let $ v $ be a B-simplicial vertex in $ (G,\psi) $. 
The following statements hold. 
\begin{enumerate}[(1)]
\item[(i)]   $ \mathbf{c}\B(G,\psi) $ is supersolvable (resp., free) if and only if $\cc\B(v)$ is supersolvable (resp.,  free). 
In this case, 
$$\exp ( \mathbf{c}\B(G,\psi)  )   = \exp  (\cc\B(v)  )  \cup \{|\psi(v)|+2e+2\ell-2\},$$ where  $ e $ denotes the number of edges incident from or on $ v $.
\item[(ii)]  If $\cc\B(v)$ is  flag-accurate whose exponents do not exceed  $|\psi(v)|+2e+2\ell-2$, then $ \mathbf{c}\B(G,\psi) $ is  flag-accurate.
\end{enumerate}
\end{proposition}

\begin{definition}
A vertex $ v $ in a digraph $ G $ is called a \textbf{king} (resp., \textbf{coking}) if $ (v,u) \in E_{G} $ (resp., $ (u,v) \in E_{G} $) for every $ u \in V_{G}\setminus\{v\} $. 
\end{definition}

 \tbf{King} and \tbf{coking elimination} operations for type $A$ were defined in \cite[Definition 2.14]{ATT21}.
Now we introduce an analog for type $B$. 

\begin{definition}
\label{def:BCE-BKE}
Let $(G,\psi)$ be a vertex-weighted digraph with nonempty interval weight,  i.e., $ \psi(i) = [a_{i}, b_{i}]\ne \varnothing$ for every $ i \in [\ell] $. Let  $v$ be a vertex in $G$.
\begin{enumerate}[(1)]
\item Suppose that $v$ is a coking in $ G $. 
The \textbf{type $B$ coking elimination (BCE)} on $G$ w.r.t. $v$ is a construction of a new vertex-weighted digraph $ (G^{\prime}, \psi^{\prime}) $ where $ G^{\prime}  =( [\ell] , E_{G'})$ is a digraph   and $ \psi^{\prime} $ is a weight given by
\begin{align*}
E_{G^{\prime}}  \coloneqq E_{G}\setminus\Set{(i,v) | i \in [\ell]\setminus\{v\}}, \qquad
\psi^{\prime}(i)  \coloneqq \begin{cases}
[a_{i}-1, b_{i}+1] & (i \in [\ell]\setminus\{v\}), \\
[a_{v}, b_{v}] & (i=v). 
\end{cases}
\end{align*}
\item Dually, suppose that  $v$ is a king in $ G $. 
The \textbf{type $B$ king elimination (BKE)} on $G$ w.r.t. $v$  produces a new vertex-weighted digraph $ (G'', \psi'') $ given by
\begin{align*}
E_{G''}  \coloneqq E_{G}\setminus\Set{(v,i) | i \in [\ell]\setminus\{v\}}, \qquad
\psi''(i)  \coloneqq \begin{cases}
[a_{i}-1, b_{i}+1] & (i \in [\ell]\setminus\{v\}), \\
[a_{v}, b_{v}] & (i=v). 
\end{cases}
\end{align*}
\end{enumerate}
\end{definition}
Let  $ G^{\mathrm{conv}} $ denote the \textbf{converse} of $G$, namely, $ G^{\mathrm{conv}} $ is the digraph obtained by reversing the direction on each edge of $ G $. 
Similar to the type $A$ case, we have $ \mathcal{B}(G,\psi) =  \mathcal{B}(G^{\mathrm{conv}}, -\psi) $ via $ x_{i} \mapsto -x_i $.
Hence taking BCE w.r.t. a coking $v$ in $(G,\psi)$ is equivalent to  taking BKE w.r.t. the king $v$ in $(G^{\mathrm{conv}}, -\psi)$.
Also, the BCE (resp., BKE) w.r.t. $v$ induces a set bijection between the associated arrangements $\ \mathcal{B}(G,\psi) \longrightarrow  \mathcal{B}(G^{\prime}, \psi^{\prime}) $ (resp., $  \mathcal{B}(G,\psi) \longrightarrow  \mathcal{B}(G'', \psi'') $).

The following analog of \cite[Proposition 7.16]{MRT23} is the main result of this subsection.
\begin{theorem}
 \label{thm:BSI-seq}
Let  $0 \le p \le \ell$.
The sequence  
$$ \B^{p,1}_\ell(m) \longrightarrow \B^{p,2}_\ell(m)\longrightarrow \cdots\longrightarrow \B^{p,\ell}_\ell(m)$$
consisting of the Shi descendants in the $p$-th row of the Shi descendant matrix and bijections between them can be constructed by BCE and the operation of adding isolated vertices on the underlying vertex-weighted digraphs. 

Moreover, for fixed $1 \le k \le \ell$ each vertex $n \in [\ell-k+2, \ell] $ is isolated and B-simplicial in  $ (T_{\ell}^{k}[n], \psi_{\ell}^{p,k}|_{[n]})$, the induced subgraph of $(T_{\ell}^{k}, \psi_{\ell}^{p,k})$ (Example \ref{ex:BSI-pda}) by $[n]=\{1,2,\ldots,n\}$.
 \end{theorem} 
 \begin{proof}
 The proof is similar in spirit to the proofs of \cite[Propositions 7.12 and 7.16]{MRT23}. 
Fix $1 \le k \le \ell-1$. 
In view of Example \ref{ex:BSI-pda}, we shall show that there exists a bijection $ \B^{p,k}_\ell(m) \longrightarrow \B^{p,k+1}_\ell(m)$ that can be constructed by taking BCE and adding isolated vertices on the underlying digraphs. 

Observe that each vertex $ i \in [\ell-k+2, \ell] $ is isolated, and its weight is the same in both $ (T_{\ell}^{k}, \psi_{\ell}^{p,k}) $ and $(T_{\ell}^{k+1}, \psi_{\ell}^{p,k+1})$.

Set $ v  \coloneqq  \ell-k+1$. 
Then $v$ is a coking of  the induced subgraph $ T_{\ell}^{k}[v] $. 
Now we show that after applying the BCE to $ (T_{\ell}^{k}[v], \psi_{\ell}^{p, k}|_{[v]}) $ w.r.t. $ v $ 
we obtain $ (T_{\ell}^{k+1}[v], \psi_{\ell}^{p,k+1}|_{[v]}) .$
Clearly, $ (T_{\ell}^{k}[v])^{\prime} = T_{\ell}^{k+1}[v] $. 
We need to show
\begin{equation*}
\label{eq:coking-v}
  (\psi_{\ell}^{p,k}|_{[v]})^{\prime}  =  \psi_{\ell}^{p,k+1}|_{[v]} .
\end{equation*}

There are two cases: $1 \leq   v \leq p$ and $p<   v \leq \ell$. 
Since the proofs are similar, we consider only the latter (which is a bit harder).
Suppose $p<   v \leq \ell$. 
Then by Definition \ref{def:BCE-BKE}(1), 
\begin{align*}
( \psi_{\ell}^{p,k}|_{[v]})^{\prime}(v) &=[1-m-\min\{\ell-v+1, k\},\,\min\{\ell-v+1, k\}+m-1]  \\
&= [1-m-k,\,k+m-1]  = \psi_{\ell}^{p,k+1}|_{[v]}(v). 
\end{align*}
Moreover, for every $ 1 \le i \le p $, 
\begin{align*}
( \psi_{\ell}^{p,k}|_{[v]})^{\prime}(i) &=[1-m-\min\{\ell-i+1, k\},\,\min\{\ell-i+1, k\}+m] \\
&=[1-m-k,\,k+m]   = \psi_{\ell}^{p,k+1}|_{[v]}(i).
\end{align*}
Similarly, for every $ p< i < v $, $( \psi_{\ell}^{p,k}|_{[v]})^{\prime}(i) =  \psi_{\ell}^{p,k+1}|_{[v]}(i)=[-m-k,\,k+m] $.

By adding the isolated vertices $v+1, \dots, \ell $ with their weights to  $ (T_{\ell}^{k}[v], \psi_{\ell}^{p,k}|_{[v]})$ and $  (T_{\ell}^{k+1}[v], \psi_{\ell}^{p,k+1}|_{[v]})$, we obtain the desired bijection 
$$ \B^{p,k}_\ell(m) \longrightarrow \B^{p,k+1}_\ell(m).$$ 

Now we show that for fixed $1 \le k \le \ell$, each isolated vertex $n \in [v+1, \ell] $ is  B-simplicial in  $ (T_{\ell}^{k}[n], \psi_{\ell}^{p,k}|_{[n]})$ by using Proposition \ref{prop:isolated}. 
Again we consider two cases: $p<   v \leq \ell$ and  $1 \leq   v \leq p$. 

If the former occurs, then 
$$
 \psi_{\ell}^{p,k}(i)  =\begin{cases}
[2 - m -k, k+m-1 ] & (1 \leq   i \leq p), \\
[ 1 - m -k, k+m-1] & (p<  i \leq  v), \\
[ - m - \ell+i , \ell-i+m ] & (v <  i \leq  \ell ).
\end{cases}
$$
 
If the latter occurs, then
$$
 \psi_{\ell}^{p,k}(i) =\begin{cases}
[2 - m -k, k+m-1] & (1 \leq   i \leq  v), \\
[ 1- m - \ell+i , \ell-i+m]  & ( v<  i \leq p), \\
[ - m - \ell+i , \ell-i+m] & (p<  i \leq  \ell ).
\end{cases}
$$

In either case, one can check that $\pm  \psi_{\ell}^{p,k}(j) \subseteq \psi_{\ell}^{p,k}(i) $ for any $v< i<j \le \ell$, and $\pm  \psi_{\ell}^{p,k}(n) \subseteq \psi_{\ell}^{p,k}(i) $ for any $ 1 \le i  \le v$. 
Now Proposition \ref{prop:isolated} completes the proof.
 
\end{proof}

Figure \ref{fig:BSI-arr} below depicts the Shi descendant  matrix for $\ell=3$ and the relation in each row from Theorem \ref{thm:BSI-seq}.

 \begin{figure}[ht!]
\centering
\begin{subfigure}{.3\textwidth}
  \centering
\begin{tikzpicture}[scale=.5]
\draw (0,1.3) node[v](1){} node[above]{\tiny $\substack{ \textbf{1} \\ [-m,m]} $};
\draw (-0.75,0) node[v](2){} node[left]{\tiny$ \substack{ \textbf{2} \\ [-m,m]} $};
\draw (0.75,0) node[v](3){} node[right]{\tiny$\substack{ \textbf{3} \\ [-m,m] }$};
\draw[>=Stealth,->] (1) to (2);
\draw[>=Stealth,->] (1) to (3);
\draw[>=Stealth,->] (2) to (3);
\end{tikzpicture}
  \caption*{$ \B^{0,1}_3 (m) $}
\end{subfigure}%
\begin{subfigure}{.3\textwidth}
  \centering
\begin{tikzpicture}[scale=.5]
\draw (0,1.3) node[v](1){} node[above]{\tiny $ \substack{ \textbf{1} \\ [-m-1,m+1]} $};
\draw (-0.75,0) node[v](2){} node[left]{\tiny $\substack{ \textbf{2} \\ [-m-1,m+1]} $};
\draw (0.75,0) node[v](3){} node[right]{\tiny $\substack{ \textbf{3} \\ [-m,m]} $};
\draw[>=Stealth,->] (1) to (2);
\end{tikzpicture}
  \caption*{$ \B^{0,2}_3 (m) $}
\end{subfigure}%
\begin{subfigure}{.3\textwidth}
  \centering
\begin{tikzpicture}[scale=.5]
\draw (0,1.3) node[v](1){} node[above]{\tiny $ \substack{ \textbf{1} \\ [-m-2,m+2]} $};
\draw (-0.75,0) node[v](2){} node[left]{\tiny $\substack{ \textbf{2} \\ [-m-1,m+1]} $};
\draw (0.75,0) node[v](3){} node[right]{\tiny $\substack{ \textbf{3} \\ [-m,m]} $};
\end{tikzpicture}
  \caption*{$ \B^{0,3}_3 (m) $}
\end{subfigure}%

\bigskip
\begin{subfigure}{.3\textwidth}
  \centering
\begin{tikzpicture}[scale=.5]
\draw (0,1.3) node[v](1){} node[above]{\tiny $\substack{ \textbf{1} \\ [1-m,m]} $};
\draw (-0.75,0) node[v](2){} node[left]{\tiny$ \substack{ \textbf{2} \\ [-m,m]} $};
\draw (0.75,0) node[v](3){} node[right]{\tiny$\substack{ \textbf{3} \\ [-m,m] }$};
\draw[>=Stealth,->] (1) to (2);
\draw[>=Stealth,->] (1) to (3);
\draw[>=Stealth,->] (2) to (3);
\end{tikzpicture}
  \caption*{$ \B^{1,1}_3 (m) $}
\end{subfigure}%
\begin{subfigure}{.3\textwidth}
  \centering
\begin{tikzpicture}[scale=.5]
\draw (0,1.3) node[v](1){} node[above]{\tiny $ \substack{ \textbf{1} \\ [-m,m+1]} $};
\draw (-0.75,0) node[v](2){} node[left]{\tiny $\substack{ \textbf{2} \\ [-m-1,m+1]} $};
\draw (0.75,0) node[v](3){} node[right]{\tiny $\substack{ \textbf{3} \\ [-m,m]} $};
\draw[>=Stealth,->] (1) to (2);
\end{tikzpicture}
  \caption*{$ \B^{1,2}_3 (m) $}
\end{subfigure}%
\begin{subfigure}{.3\textwidth}
  \centering
\begin{tikzpicture}[scale=.5]
\draw (0,1.3) node[v](1){} node[above]{\tiny $ \substack{ \textbf{1} \\ [-m-1,m+2]} $};
\draw (-0.75,0) node[v](2){} node[left]{\tiny $\substack{ \textbf{2} \\ [-m-1,m+1]} $};
\draw (0.75,0) node[v](3){} node[right]{\tiny $\substack{ \textbf{3} \\ [-m,m]} $};
\end{tikzpicture}
  \caption*{$ \B^{1,3}_3 (m) $}
\end{subfigure}%

\bigskip
\begin{subfigure}{.3\textwidth}
  \centering
\begin{tikzpicture}[scale=.5]
\draw (0,1.3) node[v](1){} node[above]{\tiny $\substack{ \textbf{1} \\ [1-m,m]} $};
\draw (-0.75,0) node[v](2){} node[left]{\tiny$ \substack{ \textbf{2} \\ [1-m,m]} $};
\draw (0.75,0) node[v](3){} node[right]{\tiny$\substack{ \textbf{3} \\ [-m,m] }$};
\draw[>=Stealth,->] (1) to (2);
\draw[>=Stealth,->] (1) to (3);
\draw[>=Stealth,->] (2) to (3);
\end{tikzpicture}
  \caption*{$ \B^{2,1}_3 (m) $}
\end{subfigure}%
\begin{subfigure}{.3\textwidth}
  \centering
\begin{tikzpicture}[scale=.5]
\draw (0,1.3) node[v](1){} node[above]{\tiny $ \substack{ \textbf{1} \\ [-m,m+1]} $};
\draw (-0.75,0) node[v](2){} node[left]{\tiny $\substack{ \textbf{2} \\ [-m,m+1]} $};
\draw (0.75,0) node[v](3){} node[right]{\tiny $\substack{ \textbf{3} \\ [-m,m]} $};
\draw[>=Stealth,->] (1) to (2);
\end{tikzpicture}
  \caption*{$ \B^{2,2}_3 (m) $}
\end{subfigure}%
\begin{subfigure}{.3\textwidth}
  \centering
\begin{tikzpicture}[scale=.5]
\draw (0,1.3) node[v](1){} node[above]{\tiny $ \substack{ \textbf{1} \\ [-m-1,m+2]} $};
\draw (-0.75,0) node[v](2){} node[left]{\tiny $\substack{ \textbf{2} \\ [-m,m+1]} $};
\draw (0.75,0) node[v](3){} node[right]{\tiny $\substack{ \textbf{3} \\ [-m,m]} $};
\end{tikzpicture}
  \caption*{$ \B^{2,3}_3 (m) $}
\end{subfigure}%

\bigskip
\begin{subfigure}{.3\textwidth}
  \centering
\begin{tikzpicture}[scale=.5]
\draw (0,1.3) node[v](1){} node[above]{\tiny $\substack{ \textbf{1} \\ [1-m,m]} $};
\draw (-0.75,0) node[v](2){} node[left]{\tiny$ \substack{ \textbf{2} \\ [1-m,m]} $};
\draw (0.75,0) node[v](3){} node[right]{\tiny$\substack{ \textbf{3} \\ [1-m,m] }$};
\draw[>=Stealth,->] (1) to (2);
\draw[>=Stealth,->] (1) to (3);
\draw[>=Stealth,->] (2) to (3);
\end{tikzpicture}
  \caption*{$ \B^{3,1}_3 (m) $}
\end{subfigure}%
\begin{subfigure}{.3\textwidth}
  \centering
\begin{tikzpicture}[scale=.5]
\draw (0,1.3) node[v](1){} node[above]{\tiny $ \substack{ \textbf{1} \\ [-m,m+1]} $};
\draw (-0.75,0) node[v](2){} node[left]{\tiny $\substack{ \textbf{2} \\ [-m,m+1]} $};
\draw (0.75,0) node[v](3){} node[right]{\tiny $\substack{ \textbf{3} \\ [1-m,m]} $};
\draw[>=Stealth,->] (1) to (2);
\end{tikzpicture}
  \caption*{$ \B^{3,2}_3 (m) $}
\end{subfigure}%
\begin{subfigure}{.3\textwidth}
  \centering
\begin{tikzpicture}[scale=.5]
\draw (0,1.3) node[v](1){} node[above]{\tiny $ \substack{ \textbf{1} \\ [-m-1,m+2]} $};
\draw (-0.75,0) node[v](2){} node[left]{\tiny $\substack{ \textbf{2} \\ [-m,m+1]} $};
\draw (0.75,0) node[v](3){} node[right]{\tiny $\substack{ \textbf{3} \\ [1-m,m]} $};
\end{tikzpicture}
  \caption*{$ \B^{3,3}_3 (m) $}
\end{subfigure}%
\caption{The Shi descendant  matrix of  type $B$  for $\ell=3$. In particular, $\Shi(B_3) = \B^{3,1}_3(1)$ and $\Ish(B_3) = \B^{3,3}_3(1)$.}
\label{fig:BSI-arr}
\end{figure}
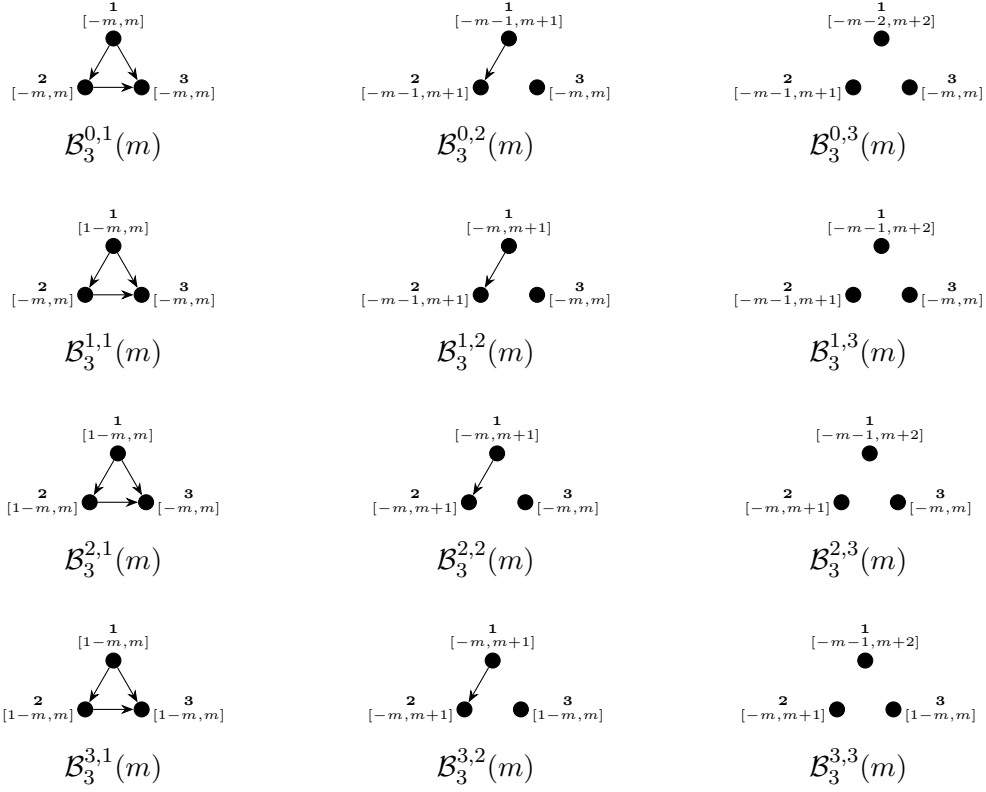

\subsection{Proof of Theorem  \ref{thm:main-BSI-free-SS}}
\label{subsec:F-SS-BSI}
\quad
First we study the flag-accuracy of the arrangements $\B^{p,1}_\ell(m)$ in the first column of the Shi descendant matrix given in \eqref{eq:1stcol}. 
Similar to the observation in Introduction for type $A$, these Shi-like arrangements contain sufficient deletions and restrictions that can guarantee not only their freeness but also their flag-accuracy. 
\begin{theorem}
 \label{thm:k=1,l,m,p}
Let $a \ge 1$, $m\ge 1$,  $\ell \ge 1$ and $0 \le p \le \ell$. 
Let $\B^p_\ell(m,a)$ be the arrangement consisting of the hyperplanes
\begin{align*}
x_{i}\pm x_{j} &= [1-a , a]\quad  (1 \leq i < j \leq \ell), \\
x_i&=  [1-m , m] \quad (1 \leq   i \leq p), \\
x_i&=[-m , m]  \quad (p<  i \leq \ell).
\end{align*}
The cone over $\B^p_\ell(m,a)$ is flag-accurate with exponents 
$$\exp(\textbf{c}\B^p_\ell(m, a)) = \{1, (2m+2a\ell-2a)^{p}, (2m+2a\ell-2a+1 )^{\ell-p}\}.$$ 

As a consequence, the Shi descendant $\B^{p,1}_\ell(m) = \B^p_\ell(m, 1) $ in the first column of the Shi descendant matrix given in \eqref{eq:1stcol} has flag-accurate cone with exponents 
$$\exp(\textbf{c}\B^{p,1}_\ell(m)) = \{1, (2m+2\ell-2)^{p}, (2m+2\ell-1)^{\ell-p}\}.$$ 
 \end{theorem} 
 
 \begin{proof}
Define the lexicographic order on the set of pairs
$$\{ (\ell, \ell-p) \mid 0\le \ell- p \le \ell, \ell \ge 1\}.$$

We first prove by induction on $ (\ell, \ell-p) $ that $\A: = \textbf{c}\B^p_\ell(m, a)$ is free with the desired exponents. 
When $\ell=1$, it is obvious. Suppose $\ell\ge2$.

Case $1$. Suppose $\ell-p\ge 1$. 
Let $H \in \A$ denote the hyperplane $x_{p+1} = -mz$. 
Then $\A' \coloneqq \A\setminus \{H\} = \textbf{c}\B^{p+1}_\ell(m,a)$ is free  with exponents 
$$\exp(\A') = \{1, (2m+2a\ell-2a)^{p+1}, (2m+2a\ell-2a+1)^{\ell-p-1}\}$$ 
by the induction hypothesis since $(\ell, \ell-p) > (\ell, \ell-p-1)$. 
Moreover, $\A'' \coloneqq \A^{H}$ consists of the following hyperplanes 
\begin{align*}
z &= 0 , \\
x_{i}\pm x_{j} &=  [1-a , a]z\quad  (1 \leq i < j \leq \ell, i\ne p+1,j\ne p+1 ), \\
x_i&=  [1-(m+a) , m+a]z \quad (1 \leq   i \leq p), \\
x_i&=[-(m+a) , m+a] z \quad (p+1<  i \leq \ell).
\end{align*}
Thus $\A'' =   \textbf{c}\B^{p}_{\ell-1}(m+a, a) $ is free with exponents 
$$\exp(\A'') = \{1, (2m+2a\ell-2a)^{p}, (2m+2a\ell-2a+1)^{\ell-p-1}\}$$ 
by the induction hypothesis since $(\ell, \ell-p) > (\ell-1, \ell-p-1)$. 
Therefore, by the addition part of Theorem \ref{thm:AD}, $\A =\textbf{c}\B^p_\ell(m, a)$ is free with the desired exponents. 

Case $2$. Suppose  $\ell=p$. 
The arrangement  in question is $\D'  \coloneqq \textbf{c}\B^\ell_\ell(m,a)$ given by
\begin{align*}
z &= 0 , \\
x_{i}\pm x_{j} &=  [1-a , a]z\quad  (1 \leq i < j \leq \ell), \\
x_i&=  [1-m , m]z \quad (1 \leq   i \leq \ell).
\end{align*}
We need to prove   that $\D'$ is free with exponents 
$$\exp(\D') = \{1, (2m+2a\ell-2a)^{\ell}\}.$$ 
Note that $\D  \coloneqq \textbf{c}\B^{\ell-1}_\ell(m,a) = \D' \cup \{K\}$, where $K  \in \D$ denotes the hyperplane $x_\ell = -mz$. 
By Case $1$, $\D$ is free with exponents 
$$\exp(\D) = \{1, (2m+2a\ell-2a)^{\ell-1}, 2m+2a\ell-2a+1\}.$$
Again by Case $1$, $\D^K= \textbf{c}\B^{\ell-1}_{\ell-1}(m+a, a)$.
Thus by the induction hypothesis, $\D^K$ is free with exponents 
$$\exp(\D^K) = \{1, (2m+2a\ell-2a)^{\ell-1}\}.$$ 
Apply the deletion part ($(1)+(3) \Rightarrow (2)$) of Theorem \ref{thm:AD}, we know that  
$\D'$ is free with the desired exponents. 

Now we prove that $\A = \textbf{c}\B^p_\ell(m, a)$ is flag-accurate by induction on $\ell$.
When $\ell=1$, it is obvious. Suppose $\ell\ge2$. 
It suffices to find an $H \in \A$ so that the restriction $\A^H$ itself is flag-accurate and its exponents are exactly the exponents of $\A$ except the largest one. 

If $\ell-p\ge 1$, by Case $1$ above we may take $H$ to be the hyperplane $x_{p+1} = -mz$. 
Then the induction hypothesis applies. 
However, when $\ell=p$, Case $2$ above does not  directly give us a candidate for the desired restriction since we have applied the deletion theorem. 
In this case, we choose $H$ to be the hyperplane $x_\ell = mz$. 
A direct computation shows $\A^H= \textbf{c}\B^{\ell-1}_{\ell-1}(m+a, a)$. 
Now  the induction hypothesis applies thanks to the calculation in Case $2$ above.

\end{proof}

Before giving the proof of  Theorem \ref{thm:main-BSI-free-SS}, we give a simple result on nonsupersolvability of the Shi descendants.

\begin{proposition}
 \label{prop:nonSS}
$ \mathbf{c}\B^{p,1}_2(m)$ is not supersolvable for any $0 \le p \le 2$, $m\ge1$.
\end{proposition}
\begin{proof}
Denote $\A \coloneqq  \mathbf{c}\B^{p,1}_2(m)$. Then $\A$  is given by 
\begin{align*}
z &= 0 , \\
x_1\pm x_2 &= [0,1]z, \\
x_i&=  [1-m , m]z \quad (1 \leq   i \leq p), \\
x_i&= [-m , m]z   \quad (p<  i \leq 2).
\end{align*}

Suppose to the contrary that $\A$ is  supersolvable. 
Note that by Theorem \ref{thm:k=1,l,m,p}, $\exp(\A) =  \{1, (2m+2)^{p}, (2m+3)^{2-p}\}$. 
By  Theorem \ref{thm:exp-ss}, there exists an M-chain $ \mathcal{A}_{1} \subseteq \mathcal{A}_{2} \subseteq \mathcal{A}_{3} = \A,$ where $\A_2$ is a modular coatom such that $|\A_{2}| \ge 2m+3$. 
However, one can check directly from the definition of $\A$ that $|\A_X| \le 2m+2$  for any coatom $X \in L(\A)$. 
This contradicts Proposition \ref{prop:BEZ}.
\end{proof}

Now we are ready to present our proof.

\begin{proof}[\tbf{Proof of  Theorem \ref{thm:main-BSI-free-SS}}]
By Theorem \ref{thm:k=1,l,m,p}, it suffices to prove that for any fixed $0 \le p \le \ell$, the cone  $\A \coloneqq \cc\B^{p,k}_\ell(m)$  for $2 \le k \le \ell$  in the $p$-th row of the Shi descendant matrix is flag-accurate with exponents  
$$\exp(\A) = \{1, (2m+2\ell-2)^{p}, (2m+2\ell-1)^{\ell-p}\} .$$ 

Set $ v  \coloneqq  \ell-k+1$. 
There are two cases: $p < v$ and  $p \ge v$. 
Since the proofs are similar, we give a proof only for the former. 
First note that by the proof of Theorem \ref{thm:BSI-seq} the arrangement $\B  \coloneqq  \textbf{c} \B(T_{\ell}^{k}[v], \psi_{\ell}^{p, k}|_{[v]})$ is given by 
\begin{align*}
z &= 0 , \\
x_{i}\pm x_{j} &= [0,1]z \quad  (1 \leq i < j \leq v), \\
x_i&=  [2-m-k , m+k-1]z \quad (1 \leq   i \leq p), \\
x_i&= [1-m-k , m+k-1]z   \quad (p<  i \leq v).
\end{align*}

Therefore, 
$$\B= \textbf{c}\B^{p,1}_{v}(m+k-1),$$
which is  flag-accurate with exponents  $  \{1, (2m+2\ell-2)^{p}, (2m+2\ell-1)^{ v-p}\}$, by Theorem \ref{thm:k=1,l,m,p}. 

Thanks to Theorem \ref{thm:BSI-seq}, each isolated vertex $n \in [v+1, \ell] $ is  B-simplicial in  $ (T_{\ell}^{k}[n], \psi_{\ell}^{p,k}|_{[n]})$ with $ \psi_{\ell}^{p,k}(n) = [-m-\ell+n, \ell-n+m]$. 
By applying Proposition \ref{prop:B-simplicial}(i) repeatedly to the $k-1$ B-simplicial vertices $\ell ,\ell-1, \ldots, v+1 $ in this order, we get  
$$ \A \mbox{ is free (resp., supersolvable)} \iff  \B  \mbox{ is free (resp., supersolvable)}  .$$  

Thus   $ \A$ is free with the desired exponents since  $| \psi_{\ell}^{p,k}(n) |+2n-2 =2m+2\ell-1 $. 
In particular, if $k=\ell$ then $v=1$. 
Therefore, $r(\B)=2$ and $\B$ is always supersolvable. 
Hence  the arrangements $\textbf{c} \B^{p,\ell}_\ell(m)$ for $0 \le p \le \ell$ in the last column of the Shi descendant matrix are supersolvable. 

Note that $\B$ is flag-accurate, none of its exponents exceeds $2m+2\ell-1 $. 
Upon applying Proposition \ref{prop:B-simplicial}(ii) repeatedly to the $k-1$ B-simplicial vertices $v+1 , \ldots, \ell $ in this order, we know that for each $n \in [v+1, \ell] $ the arrangement $\cc\B  (T_{\ell}^{k}[n], \psi_{\ell}^{p,k}|_{[n]}) $ is flag-accurate, none of its exponents exceeds $2m+2\ell-1 $. 
In particular, $ \A$ is  flag-accurate. 

Let $\ell \ge 2$. 
The nonsupersolvability of $ \mathbf{c}\B^{p,k}_\ell(m)$ for $0 \le p \le \ell$, $1 \le k < \ell$ follows from Theorem \ref{thm:localization-ss-free} and Proposition \ref{prop:nonSS}. The crucial point here is that $(1,2)$ is always an edge in $T_{\ell}^{k}$ for $ k < \ell $. 

\end{proof}

As noted in Introduction, the freeness of the type $A$ Shi descendants $\A^{p,k}_\ell(m)$ can be proved by the fact that the coking elimination under certain conditions of weights preserves freeness (and  characteristic polynomial)  \cite[Theorems 3.1 and 4.1]{ATT21}.
It would be interesting to find an analog of this result for the type $B$ coking elimination.
The computation on type $B$ is more complicated even in dimension $3$ as we will see in the next section. 
\section{Type $B$ $N$-Ish and deleted arrangements}
\label{sec:NI-deleted}

\subsection{Proof of Theorem  \ref{thm:N-Ish-B-free-SS}}
\label{subsec:F-SS-NI}
\quad 
First we study the freeness and supersolvability of rank $3$ localizations of type $B$ $N$-Ish arrangements.
\begin{proposition}
 \label{prop:N-l=2}
 Let $N = (N_1, N_2)$ be  an uneven, centered $2$-tuple  of finite sets $N_i \subseteq \Z$ (Definition \ref{def:N-cond}).
 The following are equivalent:
 \begin{enumerate}[(1)]
\item $N$ is a signed nest, i.e., $\pm N_1 \subseteq N_2$ or $\pm N_2 \subseteq N_1$.
\item  The cone $\tbf{c}\B(N)$ is supersolvable.
\item The cone $\tbf{c}\B(N)$  is  free.
\end{enumerate}
In this case, the exponents of  $\tbf{c}\B(N)$ are given by $\{1, |N_2| , |N_1| + 2\}$ if  $\pm N_1 \subseteq N_2$ and $\{1, |N_1| , |N_2| + 2\}$ if  $\pm N_2 \subseteq N_1$.
\end{proposition}

\begin{proof}
By definition, $\tbf{c}\B(N)$ consists of the following hyperplanes: 
\begin{align*}
z &= 0 ,\\
x_{1}\pm x_{2}&= 0,\\
x_{1} &= N_1z , \\
x_{2} &= N_2z . 
\end{align*}

The implication $(1) \Rightarrow (2)$ follows from Proposition \ref{prop:isolated}. 
The implication $(2) \Rightarrow (3)$ is clear from Theorem \ref{thm:exp-ss}. 
It remains to show $(3) \Rightarrow (1)$.
 
Let  $H_1,H_2,H_{\infty}$ denote the hyperplanes $x_1-x_2=0, x_1+x_2=0, z=0$, respectively. 
Set $n_1  \coloneqq |N_1|, n_2  \coloneqq  |N_2|$. 
First consider $n _1 \ge n_2$. 
Denote $ \mathcal{A}  \coloneqq \tbf{c}\B(N)$. 
The Ziegler restriction $(\A^{H_{\infty}}, m^{H_{\infty}})$ has defining polynomial 
$$Q = Q(\A^{H_{\infty}}, m^{H_{\infty}}) =x_{1}^{n_1}x_{2}^{n_2} (x_1-x_2)(x_1+x_2).$$
Using Theorem \ref{thm:criterion} we shall show that $(\A^{H_{\infty}}, m^{H_{\infty}})$ is free with  exponents $\{n_1,n_2+2\}$ and a basis
 \begin{align*}
\theta_1  & =x_{1}^{n_1} \frac{\partial}{\partial x_1} + x_{1}^{n_1-n_2-\tau(n_2)}  x_{2}^{n_2+\tau(n_2)}  \frac{\partial}{\partial x_2},\\
\theta_2  & = x_{2}^{n_2}  (x_1^2-x_2^2) \frac{\partial}{\partial x_2},
\end{align*}
where $\tau(n_2) = 0$ if $n_2$ is odd and $\tau(n_2) = 1$ otherwise. 
Note that $n_1\ge n_2+\tau(n_2) \ge 0$ by the unevenness of $N$.

Indeed, first it is not hard to check  $\theta_1, \theta_2 \in D(\A^{H_{\infty}}, m^{H_{\infty}})$. For example, 
$$\theta_1(x_{1}\pm x_{2}) =  x_{1}^{n_1-n_2-\tau(n_2)}(x_{1}^{n_2+\tau(n_2)} \pm  x_{2}^{n_2+\tau(n_2)}) \in (x_{1}\pm x_{2})\R[x_1,x_2]. $$
Moreover, 
$$
\det\begin{pmatrix}
\theta_1(x_{1}) & \theta_2(x_{1})\\
\theta_1(x_2) &\theta_2(x_2)
          \end{pmatrix}
          =
         \det \begin{pmatrix} x_{1}^{n_1} & 0\\
x_{1}^{n_1-n_2-\tau(n_2)}  x_{2}^{n_2+\tau(n_2)} &x_{2}^{n_2}  (x_1^2-x_2^2)
          \end{pmatrix}
     =Q.
       $$
       
Now suppose that $ \mathcal{A}=\tbf{c}\B(N)$ is free.
By Theorem \ref{thm:Yo-3arr}, $\exp(\A) = \{1,n_1,n_2+2\}$. 
By the deletion-restriction formula \ref{thm:DR}, 
\begin{align*}
{\chi}_\mathcal{A}(t)  & = {\chi}_{\mathcal{A} \setminus \{H_1,H_2\}}( t)-{\chi}_{(\mathcal{A} \setminus \{H_1\})^{H_2}} ( t)-{\chi}_{\mathcal{A}^{H_1}}( t).
\end{align*}
It is easily seen that $\mathcal{A} \setminus \{H_1,H_2\}$ is supersolvable with exponents $\{1, n_1,n_2\}$. 
Thus the equation above becomes 
\begin{equation*}
 \label{eq:DO}
2n_1 = |N_1 \cup N_2| + |N_1 \cup (-N_2)|.
\end{equation*}

This occurs if and only if $\pm N_2 \subseteq N_1$. 
Similarly, if $n _1 \le n_2$ then $\pm N_1 \subseteq N_2$, which completes the proof.
\end{proof}

\begin{proof}[\tbf{Proof of  Theorem \ref{thm:N-Ish-B-free-SS}}]
A similar argument as in the proof of Proposition \ref{prop:N-l=2} proves the implications $(1) \Rightarrow (2) \Rightarrow (3)$. It suffices to show $(3) \Rightarrow (1)$. 
Suppose that $ \mathcal{A}=\tbf{c}\B(N)$ is free. 
Fix $1 \leq i < j \leq \ell$, and let $X\in L(\A)$ be the subspace defined by $z=x_{i}=x_j=0$.  
Thus the localization $\A_X$ is free by Theorem \ref{thm:localization-ss-free}. 
Moreover, $\A_X$ is identical to the cone over the $N$-Ish arrangement in Proposition \ref{prop:N-l=2}. 
Hence the $2$-tuple $(N_i, N_j)$ must be a signed nest. 
Therefore $N = (N_1, \ldots, N_\ell)$ is a signed nest which completes the proof.
 \end{proof} 

  \begin{example}
\label{ex:NI-no}
We give examples showing that Theorem \ref{thm:N-Ish-B-free-SS} is no longer valid if the centeredness or unevenness is excluded.
 \begin{enumerate}[(i)]
\item The  uneven, noncentered tuple $N = (N_1,N_2,N_3)$ where $N_i =[-i,i]\setminus\{0\}$ for $1 \le i \le 3$ is a signed nest. 
However, the cone $\tbf{c}\B(N)$ is not free since its characteristic polynomial  $(t-6)(t^2-12t+39)$ does not factor completely over the integers.
\item Neither of the even, centered tuples $N = ( [0,1] , [0,1] )$ and $N' = ( [-1,2] , [-1,2] )$ is a signed nest. 
However, $\tbf{c}\B(N)$ is supersolvable, and $\tbf{c}\B(N')$ is  free but not supersolvable (see also Corollary  \ref{cor:N-l=2-integral}). 
\end{enumerate}
\end{example}

We complete this subsection by giving an explicit basis for $\tbf{c}\B(N)$ in a particular case.
  \begin{theorem}
 \label{thm:basis}
 Let $N = (N_1, \ldots, N_\ell)$ with $N_i=[-m_i,m_i]$ where $ m_1 \ge m_2 \ge \cdots \ge m_\ell \ge 0$. 
 Define the homogeneous derivations $\theta_0, \theta_1,\ldots, \theta_\ell$ as follows:
  \begin{align*}
\theta_0  & =\sum_{i=1}^\ell x_i \frac{\partial}{\partial x_i} + z\frac{\partial}{\partial z},\\
\theta_k  & =\sum_{s=k}^\ell  \left( \prod_{a \in N_k}(y_s-az) \prod_{t=1}^{k-1}(x_t^2-x_s^2) \right)\frac{\partial}{\partial x_s} \quad (1 \le k \le \ell).
\end{align*}
Then $\theta_0, \theta_1,\ldots, \theta_\ell$ form a basis for  $D(\tbf{c}\B(N))$.
 \end{theorem} 
 \begin{proof}
 Similar to the proof of \cite[Theorem 1.4]{AST17}. 
 The proof is a routine check by using Theorem \ref{thm:criterion} once we know an explicit formula of a candidate for basis. 
 \end{proof}

\subsection{Proof of Theorem  \ref{thm:SI-sym-G}}
\label{subsec:F-SS-G}
\quad
In order to characterize the freeness and supersolvability of the cone over the deleted Ish arrangement $\I(G)$ from Definition \ref{def:Del-SI}, we need to characterize these properties of the cone over an $N$-Ish arrangement of a nonnegative, centered tuple (see Proposition \ref{prop:IG-NI}).

First we study even, centered $2$-tuples. 
Compared with Proposition \ref{prop:N-l=2}, the freeness and supersolvability in this case are not always equivalent. 
   \begin{proposition}
 \label{prop:N-l=2-relax}
 Let $N = (N_1, N_2)$ be  an even, centered $2$-tuple  of finite sets $N_i \subseteq \Z$ with $|N_1| = |N_2|=n \ge 2$.
 The following are equivalent:
 \begin{enumerate}[(1)]
 \item  $|N_1 \cup (-N_1)|=n+1$, and either  $N_1 =N_2$ or  $N_1=-N_2$. 
\item The cone $\tbf{c}\B(N)$  is  free.
\end{enumerate}
In this case, the exponents of  $\tbf{c}\B(N)$ are given by $\{1, n+1 , n+1\}$. 
In addition, $\tbf{c}\B(N)$ is supersolvable if and only if either $N_1 = N_2 =\{0,a\}$, or  $N_1 = -N_2 =\{0,a\}$ for some nonzero integer $a \in \Z$. 
\end{proposition}

\begin{proof}
We will use the same notation as in Proposition \ref{prop:N-l=2}. 
Denote $ \mathcal{A}=\tbf{c}\B(N)$. 
The Ziegler restriction $(\A^{H_{\infty}}, m^{H_{\infty}})$ has defining polynomial 
$$Q = Q(\A^{H_{\infty}}, m^{H_{\infty}}) =x_{1}^{n}x_{2}^{n} (x_1-x_2)(x_1+x_2).$$

We may use Theorem \ref{thm:criterion} to check that $(\A^{H_{\infty}}, m^{H_{\infty}})$ is free with  exponents $\{n+1 , n+1\}$ and a basis
 \begin{align*}
\theta_1  & =x_{1}^{n}x_2 \frac{\partial}{\partial x_1} + x_{1}  x_{2}^{n}  \frac{\partial}{\partial x_2},\\
\theta_2  & = x_{1}^{n+1} \frac{\partial}{\partial x_1} +   x_{2}^{n+1}  \frac{\partial}{\partial x_2}.
\end{align*}
       
By the deletion-restriction formula \ref{thm:DR}, 
\begin{align*}
{\chi}_\mathcal{A}(t)  & = {\chi}_{\mathcal{A} \setminus \{H_1,H_2\}}( t)-{\chi}_{(\mathcal{A} \setminus \{H_1\})^{H_2}} ( t)-{\chi}_{\mathcal{A}^{H_1}}( t).
\end{align*}
It is easily seen that $\mathcal{A} \setminus \{H_1,H_2\}$ is supersolvable with exponents $\{1, n,n\}$. 
By Theorem \ref{thm:Yo-3arr}, $ \mathcal{A}$ is free  if and only if  the equation above (under the centeredness of $N$) becomes 
\begin{equation}
 \label{eq:E}\tag{$\star$}
2n+1 = |N_1 \cup N_2| + |N_1 \cup (-N_2)|.
\end{equation}
Thus $n \le  |N_1 \cup N_2|  \le n+1$. 
If $ |N_1 \cup N_2|=n$ then $N_1 =N_2$. 
If $ |N_1 \cup N_2| = n+1$ then $ |N_1 \cup (-N_2)|=n$ hence $N_1 =-N_2$. 
Clearly, if condition $(1)$ occurs, then ($\star$) holds true trivially.

Now suppose that  $ \mathcal{A}$ is  supersolvable. 
Hence $\A$ must be free with  $\exp(\A) = \{1,n+1,n+1\}$. 
By  Theorem \ref{thm:exp-ss}, there exists an M-chain $ \mathcal{A}_{1} \subseteq \mathcal{A}_{2} \subseteq \mathcal{A}_{3} = \A,$ where $\A_2$ is a modular coatom such that $|\A_{2}|  \ge n+2 \ge 4$. 
However, for any coatom $X \in L(\A)$ we have $|\A_X| \le 4$ or $|\A_X| =n+1$. 
Thus $n=2$ and we obtain the proposed form of $N$. 

The converse can be done by proving that $ \Cox(B_2)  \subseteq  \tbf{c}\B(N)$ is a modular coatom of $\tbf{c}\B(N)$. 
Now apply  Proposition \ref{prop:modular coatom}. 
\end{proof}

The following two corollaries are straightforward from Proposition \ref{prop:N-l=2-relax}.
 \begin{corollary}
 \label{cor:N-l=2-nonnegative}
 Let $N = (N_1, N_2)$ be  a nonnegative, even, centered $2$-tuple  of finite sets $N_i \subseteq \Z_{\ge0}$  with $|N_1| = |N_2|=n \ge 2$.
 The following are equivalent:
 \begin{enumerate}[(1)]
\item $N_1 = N_2 =\{0,a\}$ for some positive integer $a \in \Z_{>0}$.
\item  The cone $\tbf{c}\B(N)$ is supersolvable.
\item The cone $\tbf{c}\B(N)$  is  free.
\end{enumerate}
\end{corollary}

 \begin{corollary}
 \label{cor:N-l=2-integral}
 Let $N = (N_1, N_2)$ be  an even, centered $2$-tuple  of finite sets $N_i =[a_i,b_i] \subseteq \Z$ with $|N_1| = |N_2|=n \ge 2$.
 The following are equivalent:
 \begin{enumerate}[(1)]
\item Either (i) $N_1 = N_2 =[1-m,m]$ or $[-m,m-1]$, or (ii) $N_1 = -N_2 =[1-m,m]$ or $[-m,m-1]$ where $n=2m$.
\item The cone $\tbf{c}\B(N)$  is  free.
\end{enumerate}
In addition, $\tbf{c}\B(N)$ is supersolvable if and only if $\tbf{c}\B(N)$ is free and $n= 2$. 
\end{corollary}

Now we characterize the freeness and supersolvability of the cone over an $N$-Ish arrangements under centeredness and nonnegativity of  the tuple.

 \begin{theorem}
 \label{thm:N-Ish-B-free-SS-relax}
 Let $N = (N_1, \ldots, N_\ell)$ be a nonnegative, centered $\ell$-tuple  of finite sets $N_i \subseteq \Z_{\ge0}$. 
 Define $D=D(N) \coloneqq  \{i \in [\ell] \mid |N_i|>1\} \subseteq  [\ell]$ and $|N|: = \max \{  |N_i| \mid i \in D\}$.
 The following are equivalent:
 \begin{enumerate}[(1)]
\item  The cone $\tbf{c}\B(N)$ is supersolvable.
\item The cone $\tbf{c}\B(N)$  is  free.
\item One of the following conditions holds: (i) $| D | \le 1 $, or (ii)  $|D|\ge 2$ and $N_i= \{0,a\}$ for some $a \in \Z_{>0}$ and for every $i \in D$. 
\end{enumerate}
In this case, the exponents of  $\tbf{c}\B(N)$ are given by 
$$\exp(\tbf{c}\B(N)) = \{|D|+|N| -1\} \cup \{2i-1\}_{i=1}^\ell.$$ 
 \end{theorem}

  \begin{proof}
Denote  $ \mathcal{A}=\tbf{c}\B(N)$.
First we show $(3) \Rightarrow (1)$.   
If  $(3i)$ occurs, then $N$ is a signed nest that satisfies all conditions in Theorem \ref{thm:N-Ish-B-free-SS} hence $(1)$ follows immediately. 
If  $(3ii)$ occurs, then $\Cox(B_\ell)\subseteq \A$ is a modular coatom of $\A$. 
Since $\Cox(B_\ell)$ is supersolvable (e.g., by Theorem \ref{thm:N-Ish-B-free-SS}), $\A$ is also supersolvable by Proposition \ref{prop:modular coatom}. 
  
It remains to show  $(2) \Rightarrow (3)$. 
If $N$ is uneven, then by Theorem \ref{thm:N-Ish-B-free-SS} $N$ is a signed nest.
Thus there do not exist two distinct elements $N_i, N_j$ of $N$ such that $| N_i|>1$ and $| N_j |>1$, otherwise, $(N_i, N_j)$ cannot be a signed nest. 
Hence there exists at most one  $N_k$ with $| N_k |>1$, i.e., $| D | \le 1 $.

If $N$ is even, then there exist two distinct elements $N_k, N_p$ of $N$ having even cardinality with $|N_k| = |N_p| \ge 2$. 
Note that the localization $\A_X$ where $X\in L(\A)$ is the subspace defined by $z=x_{k}=x_p=0$ is free and identical to $\tbf{c}\B(N_k, N_p)$. 
By Corollary \ref{cor:N-l=2-nonnegative}, there exists $a \in \Z_{>0}$ such that $N_k = N_p =\{0,a\}$.
For any $i \in [\ell]\setminus\{k,p\}$ again by Corollary \ref{cor:N-l=2-nonnegative}, if $(N_i, N_k)$ is even, then $N_i =\{0,a\}$. 
If $(N_i, N_k)$ is uneven, then Proposition \ref{prop:N-l=2} implies $N_i =\{0\}$. 
This completes the proof.
 \end{proof}

 In what follows,  let $G=([\ell],E_G,L_G)$ be a digraph on $[\ell]$ with loop set  $L_G\subseteq[\ell]$ and edge set $E_G \subseteq \{(i,j) \mid 1 \leq i < j \leq \ell\}$.  
 
  \begin{proposition}
 \label{prop:IG-NI}
Define an $\ell$-tuple $N_G=(N_1, \ldots, N_\ell)$ by 
$$N_i  \coloneqq  \{0\} \cup  \Set{ 1 | i \in L_G} \cup  \Set{ \ell+2-j | (i,j) \in E_G} \subseteq [0,\ell].$$
Then the arrangement $\I(G)$ in Definition \ref{def:Del-SI} is identical to the $N$-Ish arrangement $\B(N_G)$ with the nonnegative, centered tuple  $N_G$.
\end{proposition}

\begin{proof}
Straightforward. 
\end{proof}

Now we characterize the freeness and supersolvability of $\tbf{c}\I(G)$.
 
   \begin{theorem}
 \label{thm:IG}
The following are equivalent:
 \begin{enumerate}[(1)]
\item  The cone $\tbf{c}\I(G)$ is supersolvable.
\item The cone $\tbf{c}\I(G)$  is  free.
\item $G$ has one of the forms described in Theorem  \ref{thm:SI-sym-G}.
\item None of the directed graphs in Figure \ref{fig:Obstructions} can occur as an induced subgraph of $G$.
\end{enumerate}
In this case, the exponents of   $\tbf{c}\I(G)$ are   given by 
$$ \exp(\tbf{c}\I(G)) = \{|E_G|+|L_G|+1\} \cup \{2i-1\}_{i=1}^\ell.$$
 \end{theorem}
 
 \begin{proof}
 The equivalence $(3) \Leftrightarrow (4)$ is not hard to verify, which depends only on the underlying digraph $G$. 
 
 To prove $(1) \Leftrightarrow (2) \Leftrightarrow (3)$, by Proposition \ref{prop:IG-NI}, it suffices to translate conditions $(3i)$ and $(3ii)$ in Theorem \ref{thm:N-Ish-B-free-SS-relax} into digraphical terms. 
Condition $(3i)$ is equivalent to one of the following 
 \begin{enumerate}
 \item[$(a')$]  there exists $k\in[\ell]$ such that $E_G=\{(k,i)  \}$ for at least one $i \in [k+1,\ell]$, $L_G =\{k\}$ ($ \{0,1\} \subsetneq N_k$ and $N_i = \{0\}$ for $i\ne k$), 
\item[$(a'')$]   there exists $k\in[\ell]$ such that $E_G=\{(k,i)  \}$ for at least one $i \in [k+1,\ell]$, $L_G =\varnothing$ ($|N_k|\ge 2$ with $1 \notin N_k$  and $N_i = \{0\}$ for $i\ne k$),
\item[$(c')$] $E_G=\varnothing$, $L_G =\{k\}$ for some $k\in[\ell]$ ($N_k = \{0,1\}$ and $N_i = \{0\}$ for $i\ne k$),  
\item[$(c'')$] $E_G=L_G =\varnothing$ ($N_i = \{0\}$ for all $i$). 
\end{enumerate}
Condition $(3ii)$ is equivalent to one of the following
 \begin{enumerate}
 \item[$(b)$]  there exists $k\in[\ell]$ such that $E_G=\{(i,k) \}$ for at least two $i$'s in $[k-1]$, $L_G =\varnothing$ 
 ($N$ has at least two distinct elements equal $\{0,a\}$ with $a \in\Z_{>1}$ and the other elements equal $\{0\}$), 
\item[$(c''')$] $E_G=\varnothing$, $|L_G|\ge2$  ($N$ has at least two distinct elements equal $\{0,1\}$  and the other elements equal $\{0\}$). 
\end{enumerate}
Summarizing the conditions above yields the desired forms of $G$.
\end{proof}

  \begin{figure}[htbp!]
\centering
\begin{subfigure}{.3\textwidth}
  \centering
\begin{tikzpicture}[scale=1]
\draw (0,0) node[v](1){} node{};
\draw (1.5,0) node[v](2){} node{};
\draw[>=Stealth,->] (1) to (2);
\Loop[dist=1cm,dir=NO,style={dashed,<-}](1);
\Loop[dist=1cm,dir=NO](2);
\end{tikzpicture}
 \end{subfigure}%
\begin{subfigure}{.3\textwidth}
  \centering
\begin{tikzpicture}[scale=1]
\draw ( 1.5, 1.5) node[v](1){} node{};
\draw (0,0) node[v](2){} node{};
\draw (0, 1.5) node[v](3){} node{};
\draw[>=Stealth,->] (3) to (1);
\draw[>=Stealth,->] (2) to (3);
\draw[>=Stealth,->,dashed] (2) to (1);
\Loop[dist=1cm,dir=WE,style={dashed,<-}](2);
\end{tikzpicture}
 \end{subfigure}%
\begin{subfigure}{.3\textwidth}
  \centering
\begin{tikzpicture}[scale=1]
\draw (0,0) node[v](1){} node{};
\draw (1.5,0) node[v](2){} node{};
\draw (0, 1.5) node[v](3){} node{};
\draw (1.5,1.5) node[v](4){} node{};
\draw[>=Stealth,->] (1) to (2);
\draw[>=Stealth,->] (3) to (4);
\draw[>=Stealth,->,dashed] (3) to (2);
\draw[>=Stealth,->,dashed] (1) to (4);
\Loop[dist=1cm,dir=WE,style={dashed,<-}](1);
\Loop[dist=1cm,dir=WE,style={dashed,<-}](3);
\end{tikzpicture}
\end{subfigure}%
\caption{Obstructions to freeness and supersolvability. Possible loops and edges are shown as dashed lines.}
\label{fig:Obstructions}
\end{figure}
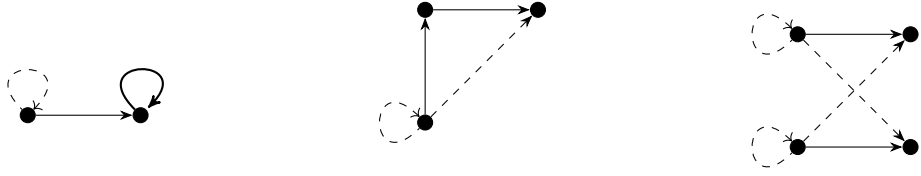

Now we characterize the freeness and supersolvability of $\cc\mcS(G)$ (Definition \ref{def:Del-SI}).
   \begin{theorem}
 \label{thm:SG}
Theorem  \ref{thm:IG} remains true if $\mcS(G)$ is put in place of $\I(G)$.
 \end{theorem}

 \begin{proof} 
It is sufficient to show $(3) \Rightarrow (1)$ and  $ (2)\Rightarrow (4)$. 
Denote $\A=\cc\mcS(G)$.
For $(3) \Rightarrow (1)$ we show that $\Cox(B_\ell)\subseteq \A$ is a modular coatom of $\mcS(G)$. 
We give a proof for one case, the remaining cases are treated similarly. 
For example, if Condition \ref{thm:SI-sym-G}(b) occurs, we may assume $E_G=\{(1,\ell),\ldots, (p,\ell)\}$ for some $p \le \ell-1$. 
Then $\A \setminus \Cox(B_\ell) =\{z=0\} \cup \{x_i - x_\ell =z \mid 1 \le i \le p\}$. 
Definition \ref{def:modco} can be checked easily. 
Now apply  Proposition \ref{prop:modular coatom}. 
 
We checked by a SageMath programing that the characteristic polynomial of any arrangement (of rank at most $4$) defined by a digraph in Figure \ref{fig:Obstructions} has a noninteger root. 
Thus the cones over the arrangements defined by these digraphs are not free by the factorization theorem \ref{thm:Factorization}. 
Using Theorem \ref{thm:localization-ss-free} we may conclude that $ (2)\Rightarrow (4)$. 
 \end{proof}

\begin{proof}[\tbf{Proof of  Theorem  \ref{thm:SI-sym-G}}]
 It follows from Theorems  \ref{thm:IG} and  \ref{thm:SG}.
 \end{proof} 

\section{Proof of Theorem  \ref{thm:1stcol-PC-intro}}
\label{sec:CQP-PC}
\quad
First we need a variant of Theorem \ref{thm:k=1,l,m,p}.
\begin{theorem}
 \label{thm:k=1,l,m,p-hat}
Let $a \ge 1$, $m\ge 0$,  $\ell \ge 1$ and $0 \le p \le \ell$. 
Let $\widehat{\B}^p_\ell(m,a)$ be the arrangement consisting of the hyperplanes
\begin{align*}
x_{i}\pm x_{j} &= [1-a , a]\quad  (1 \leq i < j \leq \ell), \\
x_i&=  [-m , m+1] \quad (1 \leq   i \leq p), \\
x_i&=[-m , m]  \quad (p<  i \leq \ell).
\end{align*}
Then the cone over $ \widehat\B^p_\ell(m,a)$
is   free with exponents 
$$\exp(\textbf{c} \widehat\B^p_\ell(m, a)) = \{1, (2m+2+2a\ell-2a)^{p}, (2m+1+2a\ell-2a)^{\ell-p}\}.$$ 
 \end{theorem} 
 
 \begin{proof}
 Similar to the proof Theorem \ref{thm:k=1,l,m,p}: Define the lexicographic order into $\{ (\ell, p) \mid 0\le p \le \ell, \ell \ge 1\}$ and prove by induction on $ (\ell,p) $.\end{proof}

The following lemma is crucial.

\begin{lemma}
 \label{lem:1stcol-PC}
If $q > 2m+2$ is an integer, then
$$\chi^{\quasi}_{\B^{p,1}_\ell(m)}(q)= \chi^{\quasi}_{\widehat\B^{\ell-p}_\ell(m,1)}(q+1).$$

 \end{lemma}
 
  \begin{proof}
  We prove the assertion by constructing a one-to-one correspondence between the complements of the $q$-reductions. 
  \begin{align*}
\chi^{\quasi}_{\B^{p,1}_\ell(m)}(q) 
&=\# \left\{ \textbf{z}  \in   \Z_q^\ell   \, \middle| \, 
\begin{array}{c}
     \overline{z}_i \pm \overline{z}_j \ne \overline0, \overline1 \quad(1 \le i <j \le \ell),  \\
      \overline{z}_i \ne \overline{c} \quad( c \in  [1-m , m],\,1 \leq   i \leq p), \\
      \overline{z}_i \ne \overline{c} \quad( c \in  [-m , m],\,p<  i \leq \ell)
    \end{array}
\right\} \\
&= \#\left\{(z_1,\ldots,z_\ell)  \in \Z^\ell  \, \middle| \,
\begin{array}{c}
     z_i - z_j \ne 0,1 \quad(1 \le i <j \le \ell),  \\
      z_i +z_j \ne q,q+1  \quad (1 \le i <j \le \ell), \\
    m+ 1 \le z_i \le q-m  \quad (1 \leq   i \leq p), \\
        m+ 1 \le z_i \le q-m-1  \quad (p<  i \leq \ell)
    \end{array}
\right\} \\
&= \#\left\{(v_1,\ldots,v_\ell)  \in \Z^\ell\, \middle| \,
\begin{array}{c}
     v_i - v_j \ne 0,1 \quad(1 \le i <j \le \ell),  \\
      v_i +v_j \ne -q,-(q+1)  \quad (1 \le i <j \le \ell), \\
    -q+m  \le v_i \le   -m- 1\quad ( \ell-p+1 \le  i \leq \ell ), \\
-q+m  +1\le v_i \le   -m- 1 \quad (1 \leq   i \leq \ell-p)
    \end{array}
\right\} \\
&= \#\left\{(t_1,\ldots,t_\ell)  \in \Z^\ell  \, \middle| \,
\begin{array}{c}
     t_i - t_j \ne 0,1 \quad(1 \le i <j \le \ell),  \\
      t_i +t_j \ne q+1,q+2  \quad (1 \le i <j \le \ell), \\
    m+ 1 \le t_i \le q-m  \quad ( \ell-p+1 \le  i \leq \ell ), \\
        m+ 2 \le t_i \le q-m  \quad (1 \leq   i \leq \ell-p)
    \end{array}
\right\} \\
&=\# \left\{ \textbf{t}  \in   \Z_{q+1}^\ell  \, \middle| \,
\begin{array}{c}
     \overline{t}_i \pm \overline{t}_j \ne \overline0, \overline1 \quad(1 \le i <j \le \ell),  \\
      \overline{t}_i \ne \overline{c} \quad( c \in  [-m , m],\, \ell-p+1 \le  i \leq \ell ), \\
      \overline{t}_i \ne \overline{c} \quad( c \in  [-m , m+1],\,1 \leq   i \leq \ell-p)
    \end{array}
\right\} \\
&=\chi^{\quasi}_{\widehat\B^{\ell-p}_\ell(m,1)}(q+1).
\end{align*} 
We have used the following changes of variables: $v_i = -z_{\ell+1-i}$, $t_i = v_i +q+1$. 

 \end{proof}

\begin{proof}[\tbf{Proof of  Theorem  \ref{thm:1stcol-PC-intro}}]
Note that the lcm periods of both $\chi^{\quasi}_{\B^{p,1}_\ell(m)}(q)$ and $\chi^{\quasi}_{\widehat{\B}^p_\ell(m,a)}(q)$   equal the lcm period of $\Cox(B_\ell)$ since the lcm period depends only on the matrix $C$. 
Thus they are equal to $2$ (see e.g., \cite[Corollary 3.2]{KTT10}). 
So these quasi-polynomials have at most two different constituents.  

By Lemma \ref{lem:1stcol-PC}, Theorems \ref{thm:KTT11} and   \ref{thm:k=1,l,m,p-hat}, if $q$ is a sufficiently large even integer, then
$$\chi^{\quasi}_{\B^{p,1}_\ell(m)}(q)  =  (q-(2m+2\ell-2))^{p}(q-(2m+2\ell-1))^{\ell-p}.$$

Thus by Theorem \ref{thm:k=1,l,m,p}, the constituents of $\chi^{\quasi}_{\B^{p,1}_\ell(m)}(q)$ are identical. 
Hence $\chi^{\quasi}_{\B^{p,1}_\ell(m)}(q)$ is a polynomial.
 
 \end{proof}

The following corollary is straightforward. 
\begin{corollary} 
\label{cor:CQP-BSI-hat}
The characteristic quasi-polynomial of $ \widehat\B^{p}_\ell(m,1)$ for each $0 \le p \le \ell$ is a polynomial and given by 
$$\chi^{\quasi}_{\widehat\B^{p}_\ell(m,1)}(t)= \chi_{\widehat\B^{p}_\ell(m,1)}(t) =  (t-(2m+2\ell))^{p}(t-(2m+2\ell-1))^{\ell-p}.$$
Hence period collapse also occurs in this case.
\end{corollary}

For the Ish arrangements in the last column of the Shi descendant matrix when $\ell=2$, we have verified that for $0\le p \le 2$
$$\chi^{\quasi}_{\B^{p,2}_2(m)}(q) = \begin{cases}
\chi_{\B^{p,2}_2(m)}(q)  & \mbox{ if $q$ is odd}, \\
\chi_{\B^{p,2}_2(m)}(q) +1 &  \mbox{ if $q$ is even}. 
\end{cases}$$

Thus the type $B$ Shi and Ish arrangements do not share the same period collapse property in general.

\vskip 1em
\noindent
\textbf{Acknowledgements.} 
 The first author was supported by a postdoctoral fellowship of the Alexander von Humboldt Foundation at Ruhr-Universit\"at Bochum.
This work started when the first author was a participant at the Combinatorial Coworkspace in March 2022. He thanks the organizers of the event for creating a collaborative research environment and Galen Dorpalen-Barry for stimulating discussion.

\bibliographystyle{amsplain}
\bibliography{references}

\end{document}